\documentclass{compositio}
\usepackage{amsmath, amssymb, latexsym, amsfonts, amscd, amsthm, xypic, setspace, color, hyperref, graphicx}

\theoremstyle{plain}
\newtheorem{theorem}{Theorem}[subsection]
\newtheorem{proposition}[theorem]{Proposition}
\newtheorem{lemma}[theorem]{Lemma}
\newtheorem{corollary}[theorem]{Corollary}

\theoremstyle{definition}
\newtheorem{definition}[theorem]{Definition}

\theoremstyle{remark}
\newtheorem{remark}[theorem]{Remark}
\newtheorem{example}[theorem]{Example}

\newcommand{\Gal}{\text{Gal}}
\newcommand{\Pic}{\text{Pic}}
\newcommand{\Piczero}{\text{Pic}^{0}}
\newcommand{\Div}{\text{Div}}

\newcommand{\xbar}{\bar{X}}

\newcommand{\Br}{\text{Br}}
\newcommand{\Hom}{\text{Hom}}

\newcommand{\Ext}{\text{Ext}}

\newcommand{\Spec}{\text{Spec}}

\newcommand{\Z}{\mathbf{Z}}
\newcommand{\Q}{\mathbf{Q}}

\newcommand{\C}{\mathbf{C}}
\newcommand{\QZ}{\mathbf{Q}/\mathbf{Z}}

\newcommand{\G}{\mathbf{G}}

\begin{document}

\title{Weil-\'etale Cohomology over $p$-adic Fields}
\author{David A. Karpuk}
\email{karpuk@math.umd.edu}
\address{Department of Mathematics, University of Maryland, College Park, MD 20910}


\begin{abstract}


We establish duality results for the cohomology of the Weil group of a $p$-adic field, analogous to, but more general than, results from Galois cohomology.  We prove a duality theorem for discrete Weil modules, which implies Tate-Nakayama Duality.  We define Weil-smooth cohomology for varieties over local fields, and prove a duality theorem for the cohomology of $\G_m$ on a smooth, proper curve with a rational point.  This last theorem is analogous to, and implies, a classical duality theorem for such curves.
\end{abstract}

\maketitle


\section{Introduction}
\label{sec:introduction}

\subsection{Background and Motivation}

Arithmetic applications of Weil groups have been a popular topic in recent years, starting with the article \cite{lichtenbaum2005} of Lichtenbaum.  In this article, Lichtenbaum defines a cohomology theory for varieties over finite fields, Weil-\'etale cohomology, wherein the Weil group plays the role that the Galois group plays in \'etale cohomology.  The Weil-\'etale cohomology groups of the sheaf $\Z$ on a smooth, projective variety are shown to be finitely generated abelian groups.  The resulting secondary Euler characteristics provide a cohomological interpretation of the order of vanishing and leading coefficient of the zeta function of the variety at $t=1$.

If $k$ is a finite field, the \'etale and Weil-\'etale cohomology groups of $X=\Spec\ k$ with $\Z$ coefficients are given, respectively, by
$$
H^i(\hat{\Z},\Z)=\left\{ \begin{array}{cl}
\Z & \text{if $i=0$}\\
0 & \text{if $i=1$}\\
\QZ & \text{if $i=2$}\\
0 & \text{if $i\geq 3$}
\end{array}
\right.
\qquad \text{and} \qquad
H^i(\Z,\Z)=\left\{ \begin{array}{cl}
\Z & \text{if $i=0$}\\
\Z & \text{if $i=1$}\\
0 & \text{if $i\geq 2.$}
\end{array}
\right.
$$
In general, taking Weil-\'etale cohomology of $\Spec\ k$ shifts the $\QZ$'s that appear in the \'etale cohomology groups down a degree and turning them into $\Z$'s.  In this sense, Weil-\'etale cohomology of $\Spec\ k$ determines the \'etale cohomology, as is made precise by Lemma 1.2 of \cite{lichtenbaum2005}.

Let $K$ be a $p$-adic local field with absolute Galois group $G$ and Weil group $W$.  Let $L$ be the completion of the maximal unramified extension of $K$, and let $\bar{L}$ be an algebraic closure of $L$ containing an algebraic closure $\bar{K}$ of $K$.  If $X=\Spec\ K$, the \'etale and Weil-\'etale cohomology groups of $X$ with $\G_m$ coefficients are given, respectively, by
$$
H^i(G,\bar{K}^\times)=\left\{ \begin{array}{cl}
K^\times & \text{if $i=0$}\\
0 & \text{if $i=1$}\\
\QZ & \text{if $i=2$}\\
0 & \text{if $i\geq 3$}
\end{array}
\right.
\qquad \text{and} \qquad
H^i(W,\bar{L}^\times)=\left\{ \begin{array}{cl}
K^\times & \text{if $i=0$}\\
\Z & \text{if $i=1$}\\
0 & \text{if $i\geq 2$.}
\end{array}
\right.
$$
For any connected, commutative algebraic group $\mathcal{A}/K$, the groups $H^i(W,\mathcal{A}(\bar{L}))$ determine the groups $H^i(G,\mathcal{A}(\bar{K}))$ up to isomorphism, which is made precise by our Theorem (\ref{weilgalalgcomp}).

Lichtenbaum's computation of the groups $H^i(X_W,\Z)$ for a curve $X$ over a finite field, and therefore his interpretation of special values of zeta functions, relies on a duality theorem stated in terms of cup-product in Weil-\'etale cohomology.  The main theorems presented here are analogous duality theorems for the Weil-\'etale cohomology of zero and one-dimensional schemes over $p$-adic fields.  The former is the Weil analogue of Tate-Nakayama Duality, and the latter is the Weil analogue of Lichtenbaum Duality for curves over $p$-adic fields.

Our main theorem concerning the cohomology of $W$-modules is Theorem (\ref{fgweilduality}).  Let $\Gamma_W:W$-Mod$\rightarrow \mathcal{A}b$ be the fixed-points functor, and let $R\Gamma_W:\mathcal{D}(W)\rightarrow \mathcal{D}(\Z)$ be its derived functor.  For a $W$-module $M$, let $M^D=\Hom(M,\bar{L}^\times)$.  There is a natural map in $\mathcal{D}(\Z)$,
$$
\psi(M):R\Gamma_W(M^D)\rightarrow R\Hom(R\Gamma_W(M),\Z[-1]),
$$
induced by a cup-product pairing.  Our theorem is the following:

\noindent \textbf{Theorem \ref{fgweilduality}.}  Suppose that $M$ is finitely generated as an abelian group.  Then $\psi(M)$ is an isomorphism in $\mathcal{D}(\Z)$. \hfill \qed

This theorem implies Tate-Nakayama Duality, which in turn implies the main theorem of Local Class Field Theory.  This theorem was originally proven by Jiang in \cite{jiang2006} using Tate-Nakayama Duality, under the assumption that $M$ is a $G$-module.  We have removed this assumption, and provided a proof which is independent of the main results of Galois cohomology.  

Our second main theorem is a duality theorem for the Weil-\'etale cohomology of smooth, projective, geometrically connected curves $X/K$, which contain a rational point.  Let $\Gamma_X$ be the global sections functor on the Weil-\'etale site of $X$, and let $R\Gamma_X$ be the derived functor of $\Gamma_X$.  Our duality theorem for $X$ is the following:

\noindent \textbf{Theorem \ref{curveduality}.} Let $X/K$ be a smooth, projective, geometrically connected curve over $K$, such that $X(K)\neq\emptyset$.  There is a symmetric pairing
$$R\Gamma_X(\G_m)\otimes^LR\Gamma_X(\G_m)\rightarrow \Z[-2],$$
such that the induced map
$$
R\Gamma_X(\G_m)\rightarrow R\Hom(R\Gamma_X(\G_m),\Z[-2])
$$
is an isomorphism on cohomology in degree $0$ and $1$, and injective on cohomology in degree $2$ and $3$.  The cohomology of both complexes vanishes outside of degrees $0$ through $3$. \hfill \qed

Our theorem is stated and proved more naturally in the setting of Weil-smooth cohomology, which is defined in the last chapter.  However, Weil-smooth and Weil-\'etale cohomology groups agree when the sheaf is given by a smooth, commutative group scheme, just as smooth and \'etale cohomology agree for such sheaves.

What prevents the map from being an isomorphism on the nose is that one essentially encounters the inclusion map $U_K^*\rightarrow \Hom(U_K,\QZ)$, which is not surjective.  This self-duality of $\G_m$ is the Weil analogue of the classical duality theorems of the article \cite{lichtenbaum1969}.  In fact, for curves with rational points, one can deduce the main result of \cite{lichtenbaum1969}, namely that there is a natural isomorphism $\Br(X)\rightarrow \Pic(X)^*$.

More recent work by Lichtenbaum \cite{lichtenbaum2009}, Flach \cite{flach2008}, Morin \cite{morin20101} and \cite{morin20102}, and Flach and Morin \cite{flachmorin2010} has been done towards a definition of Weil-\'etale cohomology for schemes of finite type over $\Spec\ \Z$.  These approaches have been partially successful in giving a Weil-\'etale cohomological interpretation of special values of zeta functions and $L$-functions of such schemes, but a fully satisfying global theory is still lacking.  Hopefully a better understanding of Weil-\'etale cohomology over $p$-adic fields will help provide a link between the respective theories over finite and  global fields.

\subsection{Notation}
If $R$ is a commutative ring, then $\mathcal{D}(R)$ denotes the bounded derived category of $R$-modules.  If $G$ is a discrete group, then $\mathcal{D}(G)$ denotes the bounded derived category of $G$-modules; if $G$ is profinite, then we will always restrict our attention to discrete, continuous $G$-modules, unless stated otherwise.  If $G$ is a discrete group and $M$ is a $G$-module, then by $H^*(G,M)$ we mean the traditional group cohomology of $G$ with coefficients in $M$.  If $G$ is profinite, or an extension of a discrete group by a profinite group, and $M$ is a discrete $G$-module, then by $H^*(G,M)$ we mean Galois cohomology in the sense of \cite{serre2002}.

We briefly recall the notion of cohomological dimension.  Let $G$ be a discrete group, or an extension of a discrete group by a profinite group.  The \textit{cohomological dimension} of $G$ is defined to be the smallest integer $n$ such that for all $m>n$ and all torsion $G$-modules $M$, we have $H^m(G,M)=0$ (provided such an $n$ exists).  We write $\text{cd}(G)$ for the cohomological dimension of $G$.  The \textit{strict cohomological dimension} of $G$ is defined to be the smallest integer $n$ such that for all $m>n$ and all $G$-modules $M$, we have $H^m(G,M)=0$ (provided such an $n$ exists); we denote the strict cohomological dimension of $G$ by $\text{scd}(G)$.

If $f:X\rightarrow Y$ is a map between two cochain complexes, we use $f^i$ to denote the induced map in $i^{th}$ cohomology.  We will often write an exact triangle $X\rightarrow Y\rightarrow Z\rightarrow X[1]$ as simply $X\rightarrow Y\rightarrow Z$, with the $X[1]$ being implied.

If $M$ is an abelian group and $n$ is an integer, we use $M[n]$ and $M/n$ to denote the kernel and cokernel, respectively, of the multiplication-by-$n$ maps on $M$.  If $M$ and $N$ are abelian groups with some possible extra structure (for example, $M$ and $N$ could be $G$-modules), then $\Hom(M,N)$ will always mean $\Hom_\Z(M,N)$, and similarly for $\Ext$ and $\otimes$.  When working in the derived category of abelian groups, we will write $R\Hom$ for $R\Hom_{\mathcal{D}(\Z)}$.  If $X$ and $Y$ are objects in $\mathcal{D}(\Z)$, we will write $\Ext^i(X,Y)$ for their $i^{th}$ hyperext group, see Definition 10.7.1 of \cite{weibel1994}.

For a topological abelian group $A$, we define $A^*=\Hom_{\text{cont}}(A,\QZ)$, where $\QZ$ has the discrete topology.  If $A$ is the group of rational points of some commutative algebraic group scheme over a local field, it will be understood that $A$ is endowed with the natural topology coming from the local field.

If $\frak{X}$ is any Grothendieck site, we denote by $\mathcal{S}(\frak{X})$ the category of sheaves of abelian groups on $\frak{X}$.  We let $\mathcal{D}(\frak{X})$ denote the bounded derived category of $\mathcal{S}(\frak{X})$.

\section{Weil Groups of Finite Fields}
\label{sec:Weil Groups of Finite Fields}

We begin with a duality theorem for the cohomology of the Weil group $\frak{w}$ of a finite field $k$, and restate the theorem as a duality theorem in Weil-\'etale cohomology.  Since $\frak{w}\simeq \Z$ as abstract groups, this is mostly an exercise in homological algebra.  It will be useful to have at the ready the following explicit description of the $\Z$-dual of a complex.

\begin{proposition}\label{rhomz} Let $C$ be a bounded cochain complex of abelian groups, considered as an object in $\mathcal{D}(\Z)$.  Let $\Z[-n]$ be the cochain complex with $\Z$ in degree $n$ and $0$ everywhere else.  Then for all $i$ we have short exact sequences
\begin{equation}
0\rightarrow \Ext(H^{n-i+1}(C),\Z)\rightarrow \Ext^i(C,\Z[-n])\rightarrow \Hom(H^{n-i}(C),\Z)\rightarrow 0.
\end{equation}
\end{proposition}

\begin{proof} This follows easily by applying $R\Hom(C,-)$ to the exact sequence $0\rightarrow \Z\rightarrow \Q\rightarrow \QZ\rightarrow 0$ and taking cohomology.
\end{proof}

\subsection{Duality for Finitely Generated Modules}

\begin{proposition}\label{cupproduct}  Let $R$ be a commutative ring with $1$, and suppose that $M$ is an $R[\frak{w}]$-module.  Let $D$ be an injective $R$-module on which $\frak{w}$ acts trivially, so that $H^{1}( \frak{w},D)=D$.  Then the cup-product pairing
$$H^{i}( \frak{w},M)\otimes_{R} H^{1-i}( \frak{w},\Hom_{R}(M,D))\rightarrow D$$
induces isomorphisms $H^{1-i}( \frak{w},\Hom_{R}(M,D))\simeq \Hom_{R}(H^{i}( \frak{w},M),D)$ for $i=0,1$.
\end{proposition}
\begin{proof}  Applying the functor $\Hom_{R}(-,D)$ to the exact sequence $0\rightarrow M^{ \frak{w}}\rightarrow M\stackrel{\sigma-1}{\rightarrow}M\rightarrow M_{ \frak{w}}\rightarrow 0$ we obtain the exact sequence
$$0\rightarrow \Hom_{R}(M_{ \frak{w}},D)\rightarrow \Hom_{R}(M,D)\stackrel{\sigma-1}{\rightarrow}\Hom_{R}(M,D)\rightarrow \Hom_{R}(M^{ \frak{w}},D)\rightarrow 0.$$
The proposition now follows from the fact that $H^{0}( \frak{w},\Hom_{R}(M,D))$ and $H^{1}( \frak{w},\Hom_{R}(M,D))$ are, respectfully, the kernel and cokernel of the map $\sigma-1$ from $\Hom_{R}(M,D)$ to itself.
\end{proof}

\begin{example}\label{gamma0finiteduality} Let $R=\Z$, $D=\QZ$, and suppose that $M$ is finite.  The canonical isomorphism $M=M^{**}$ yields the following perfect pairings of finite abelian groups:
$$
H^{i}( \frak{w},M)\otimes H^{1-i}( \frak{w},M^*)\rightarrow \QZ.
$$
\end{example}

\begin{example}\label{gamma0vsduality} Let $R$ be a field $F$ of characteristic zero, and suppose that $M=V$ is a finite-dimensional representation of $ \frak{w}$.  Let $D=F$ be the trivial representation, and consider the dual representation $\Hom_F(V,D)$.  By (\ref{cupproduct}) we have the following perfect pairings of vector spaces over $F$:
$$
H^i(\frak{w},V)\otimes_F H^{1-i}(\frak{w},\Hom_F(V,F))\rightarrow F.
$$
\end{example}

Suppose that $M$ is a bounded complex of $\frak{w}$-modules, considered as an object in $\mathcal{D}(\frak{w})$.  Let us define its \textit{dual complex} by $M^D:= R\Hom(M,\Z)\in\mathcal{D}(\frak{w})$, which comes equipped with a canonical pairing $M\otimes^L M^D\rightarrow \Z$ in $\mathcal{D}(\frak{w})$.  
The projection map $R\Gamma_{\frak{w}}(\Z)\rightarrow\Z[-1]$ gives a cup-product pairing
$$
R\Gamma_{\frak{w}}(M)\otimes^L R\Gamma_{\frak{w}}(M^D)\rightarrow \Z[-1]
$$
which induces a map
\begin{equation}\label{zcup}
\psi(M):R\Gamma_{\frak{w}}(M^D)\rightarrow R\Hom(R\Gamma_{\frak{w}}(M),\Z[-1]).
\end{equation}

\begin{proposition}\label{finite} Suppose that $M$ is finite.  Then $\psi(M)$ is an isomorphism.
\end{proposition}
\begin{proof}  Because $M$ is finite, we have identifications $M^D=M^*[-1]$ and $\Ext^i(R\Gamma_{\frak{w}}(M),\Z[-1])=H^{2-i}(\frak{w},M)^*,$ the second of which follows from (\ref{gamma0finiteduality}).  The map $\psi(M)^i$ is the map $H^{i-1}(\frak{w},M^*)\rightarrow H^{2-i}(\frak{w},M)^*$ induced by cup-product, which is an isomorphism by (\ref{gamma0finiteduality}).
\end{proof}

\begin{proposition}\label{free}  Suppose that $M$ is free and finitely generated as an abelian group.  Then $\psi(M)$ is an isomorphism.  
\end{proposition}
\begin{proof}
(The author thanks Thomas Geisser for pointing out this argument, which is presented in Lemma 7.5 of \cite{geisser2010}.)  Let $L(-)_{\frak{w}}$ be the derived functor of $M\mapsto M_{\frak{w}}$.  There is a natural equivalence $L(-)_{\frak{w}}\simeq R\Gamma_{\frak{w}}(-)[1]$.  Combined with the adjunction map $\Hom_{\frak{w}}(M,\Z)\simeq \Hom(M_{\frak{w}},\Z)$ induced by the evaluation pairing, taking derived functors yields the isomorphisms
$$
R\Hom_{\frak{w}}(M,\Z)\simeq R\Hom(L(M)_\frak{w},\Z)\simeq R\Hom(R\Gamma_{\frak{w}}(M)[1],\Z)\simeq R\Hom(R\Gamma_{\frak{w}}(M),\Z[-1])
$$
which gives the result.
\end{proof}

\begin{theorem}\label{fggamma0duality}  Let $M$ be a bounded complex of $ \frak{w}$-modules, whose cohomology groups are finitely generated as abelian groups.  Then the map $\psi(M)$ of (\ref{zcup}) is an isomorphism.
\end{theorem}
\begin{proof}
This follows easily from the above two propositions; one first proves it first for $M$ concentrated in degree $0$ by considering the exact sequence $0\rightarrow M_{\text{tors}}\rightarrow M\rightarrow M/M_{\text{tors}}\rightarrow 0$.  The general result follows by induction on the length of the complex.
\end{proof}


\subsection{Weil-\'etale Cohomology over Finite Fields}

Let us restate the duality theorem of the last section in terms of Weil-\'etale cohomology.  For a scheme $X$ which is finite type over $k$, let $X_W$ denote $X$ endowed with the Weil-\'etale topology, as defined in \cite{lichtenbaum2005}.  Let $\Gamma_X:\mathcal{S}(X_W)\rightarrow \mathcal{A}b$ be the global sections functor, and let $R\Gamma_X:\mathcal{D}(X_W)\rightarrow \mathcal{D}(\Z)$ be its derived functor.

Now let $X=\Spec\ k$.  By Proposition 2.2 of \cite{lichtenbaum2005}, Weil-\'etale sheaves on $X_W$ are simply $\frak{w}$-modules.  Theorem (\ref{fggamma0duality}) has the following rephrasing in terms of Weil-\'etale cohomology:

\begin{theorem}
Let $F$ be a bounded complex of Weil-\'etale sheaves on $X=\Spec\ k$, such that the cohomology sheaves $H^i(F)$ correspond to finitely generated $\frak{w}$-modules.  Let $F^D = R\underline{\Hom}_X(F,\Z)$.  Then the cup-product pairing
$$
R\Gamma_X(F)\otimes^L R\Gamma_X(F^D)\rightarrow \Z[-1]
$$
induces an isomorphism
\begin{equation}
R\Gamma_X(F^D)\rightarrow R\Hom(R\Gamma_X(F),\Z[-1])
\end{equation}
in the derived category of abelian groups.\hfill\qed
\end{theorem}

We remind the reader of Lichtenbaum's duality theorem for the Weil-\'etale cohomology of curves, proved  in \cite{lichtenbaum2005}.  Let $X/k$ be a smooth, geometrically connected curve.  For simplicity we assume $X$ is projective, though Lichtenbaum does not make this assumption.  One has $H^2(X_W,\G_m)=\Z$ and $H^i(X_W,\G_m)=0$ for $i\geq 3$.

\begin{theorem} (\cite{lichtenbaum2005}, Theorem 5.1)
Let $F$ be a locally constant Weil-\'etale sheaf on $X$, representable by a finitely generated abelian group, and let $F^D=R\underline{\Hom}_X(F,\G_m)$.  The cup-product pairing
$$
R\Gamma_X(F)\otimes^LR\Gamma_X(F^D)\rightarrow \Z[-2]
$$
induces an isomorphism
\begin{equation}
R\Gamma_X(F^D)\rightarrow R\Hom(R\Gamma_X(F),\Z[-2])
\end{equation}
in the derived category of abelian groups.  \hfill\qed
\end{theorem}

\subsection{Some Finiteness and Vanishing Lemmas}

The following two lemmas will prove useful in the next chapter, and are generalizations of the additive and multiplicative Hilbert Theorem 90 for finite fields.

\begin{lemma}\label{H0finite}  Let $M$ be a $ \frak{w}$-module which is finitely generated and free as an abelian group.  Then:
\begin{itemize}
\item[(i)] $H^{0}( \frak{w},\Hom(M,\bar{k}^{\times}))$ is finite,
\item[(ii)] $H^{0}( \frak{w},\Hom(M,\bar{k}))$ is finite,
\item[(iii)] and $H^{0}( \frak{w},\Hom(M,U_{L}))$ is profinite.
\end{itemize}
\end{lemma}
\begin{proof}

(i) Let $f:M\rightarrow\bar{k}^{\times}$ be $ \frak{w}$-equivariant.  Then $f(\sigma m)=f(qm)$, so $f$ factors through the finite group $M/(\sigma-q)M$.  Since $\bar{k}^{\times}$ has only finitely many elements of any particular order, the image of $f$ lands in a finite set which is independent of $f$.  Therefore there are only finitely many possible $f$.

(ii) Choose a $\Z$-basis of $M$ and a matrix $A$ representing the action of $\sigma$ on $M\cong \Z^{r}$.  This $\Z$-basis provides us with an isomorphism $\Hom(M,\bar{k})\simeq \bar{k}^{r}$.  Under this isomorphism, elements of $H^{0}( \frak{w},\Hom(M,\bar{k}))$ correspond to solutions of the equation $Ax=x^{q}$, for $x\in\bar{k}^{r}$ (here raising to the $q^{th}$ power is done component-wise). 

 Writing the coordinates of $x$ as $X_{1},\ldots,X_{r}$, we see that we must show that the affine variety $V$ defined by the equations $\sum_{i=1}^{r}a_{ij}X_{i}-X_{j}^{q}=0$ for $1\leq j\leq r$ is finite, where $A=(a_{ij})$.  Let $V_j\subset \bar{k}^r$ be the hypersurface defined by $\sum_{i=1}^{r}a_{ij}X_{i}-X_{j}^{q}$, so that $V=\bigcap_j V_j$.  Let $v\in V$, and let $T_v$ be the tangent space to $V$ at $v$.  Then $0 = \dim \ker(A) = \dim T_v\geq \dim V,$ hence $V$ has dimension zero and is therefore finite.

(iii) Since $H^0(\frak{w},-)$ commutes with inverse limits, we have
$$H^{0}( \frak{w},\Hom(M,U_{L})) = \varprojlim_{i} H^{0}( \frak{w},\Hom(M,U_{L}/U^{(i)}_{L})).$$
Inducting on the sequences $0\rightarrow U^{(i)}_{L}/U^{(i+1)}_{L}\rightarrow U_{L}/U^{(i+1)}_{L}\rightarrow U_{L}/U^{(i)}_{L}\rightarrow 0$, we see from parts (i) and (ii) that the terms appearing in the above inverse limit are all finite, hence the result.\qedhere
\end{proof}

\begin{lemma}\label{H1vanish}  Let $M$ be a $ \frak{w}$-module which is finitely generated and free as an abelian group.  Then:
\begin{itemize}
\item[(i)]  $H^{1}( \frak{w},\Hom(M,\bar{k}^{\times}))=0$.
\item[(ii)] $H^{1}( \frak{w},\Hom(M,\bar{k}))=0$.
\item[(iii)] $H^{1}( \frak{w},\Hom(M,U_{L}))=0$.  (See Chapter XIII, Proposition 15 of \cite{serre1979} for the case $M=\Z$.)
\end{itemize}
\end{lemma}
\begin{proof} 
(i) Writing the group law on $\bar{k}^{\times}$ additively, we wish to show that the map $(\sigma-1):\Hom(M,\bar{k}^{\times})\rightarrow \Hom(M,\bar{k}^{\times})$ is surjective.  Choosing a $\Z$-basis of $M$ determines an isomorphism $\Hom(M,\bar{k}^{\times})\simeq (\bar{k}^{\times})^{r}$ and a matrix $A\in \text{GL}_{r}(\Z)$ representing the action of $\sigma$ on $M$.  Direct calculation shows that under our isomorphism $\Hom(M,\bar{k}^{\times})\simeq (\bar{k}^{\times})^{r}$, the map $\sigma-1$ is transformed into $qA^{-1}-1:(\bar{k}^{\times})^{r}\rightarrow (\bar{k}^{\times})^{r}$.  Multiplying by the automorphism $-A$, we conclude that it suffices to prove that $A-q$ is surjective, as a map from $(\bar{k}^{\times})^{r}$ to itself.

It is clear that $q$ is not an eigenvalue of $A$, since $\det(A)=\pm 1$, but the characteristic polynomial of $A$ has integer coefficients.  It follows that the map $A-q:\Z^{r}\rightarrow \Z^{r}$ is injective with finite cokernel $C$.  Tensoring with $\bar{k}^{\times}$, we obtain an exact sequence $(\bar{k}^{\times})^{r}\stackrel{A-q}{\rightarrow}(\bar{k}^{\times})^{r}\rightarrow C\otimes\bar{k}^{\times}\rightarrow 0.$  But $C\otimes\bar{k}^{\times}=0$, because $\bar{k}^{\times}$ is divisible and $C$ is finite.

(ii) As in part (i), we choose a $\Z$-basis for $M$ and a matrix $A$ representing the action of $\sigma$ on $M$.  On the group $\Hom(M,\bar{k})$, direct calculation shows that $((\sigma-1)f)(m)=f(A^{-1}m)^{q}-f(m)$, and replacing $m$ with $Am$ and multiplying by $-1$, we are reduced to showing that the map $f\mapsto f\circ A- f^{q}$ is surjective.  Under the isomorphism $\Hom(M,\bar{k})\simeq \bar{k}^{r}$, this is the map $x\mapsto Ax-x^{q}$, where raising to the $q^{th}$ power is done component-wise.  A proof of the surjectivity of this map is contained in Corollary 5.1.2 of \cite{haines}.


(iii) Recall that $U_{L}$ comes equipped with a filtration $\cdots\subset U_{L}^{(i)}\subset\cdots\subset U_{L}^{(1)}\subset U_{L}$ of $ \frak{w}$-modules, with $U_{L}/U_{L}^{(1)}\simeq \bar{k}^{\times}$, and higher successive quotients all isomorphic to $\bar{k}$.  Let $f:M\rightarrow U_{L}$, and reduce modulo $U_{L}^{(1)}$.  We obtain a map $\bar{f}:M\rightarrow \bar{k}^{\times}$, and by part (i) we can write $\bar{f}=(\sigma-1)\bar{f}_{0}$ for some $\bar{f}_{0}:M\rightarrow\bar{k}^{\times}$.  As $M$ is free, we can lift $\bar{f}_{0}$ to a map $f_{0}:M\rightarrow U_{L}$, and we have $f=(\sigma-1)f_{0}+g_{1}$ for some $g_{1}:M\rightarrow U_{L}^{(1)}$.  
It is clear that by repeating this process, we can write $f=\sum_{i\geq 0}(\sigma-1)f_{i}=(\sigma-1)(\sum_{i\geq0}f_{i})$, where $f_{i}:M\rightarrow U_{L}^{(i)}$. \qedhere
\end{proof}

\section{Weil Groups of $p$-adic Fields}
\label{sec:Weil Groups of $p$-adic Fields}

We now turn our attention to $p$-adic fields.  The main theorem of this chapter is a duality theorem for the cohomology of the Weil group of a $p$-adic field, which can be interpreted as a duality theorem for the Weil-\'etale cohomology of the spectrum of the field.

Let $K$ be a local field of characteristic zero, with finite residue field $k$. Let $K_{ur}$ be its maximal unramified extension, and let $L$ be the completion of $K_{ur}$.  Let $\bar{K}$ be an algebraic closure of $K$, and let $\bar{L}$ be an algebraic closure of $L$ containing $\bar{K}$.

Let $\frak{g}=\text{Gal}(K_{ur}/K)=\text{Gal}(\bar{k}/k)\simeq \hat{\mathbf{Z}}$, which has subgroup $\frak{w}\simeq \mathbf{Z}$ consisting of integral powers of the Frobenius.   The action of $\frak{w}$ on $K_{ur}$ extends by continuity to $L$.  We let $I=\text{Gal}(\bar{K}/K_{ur})$ be the inertia subgroup of $G=\text{Gal}(\bar{K}/K)$.  It follows from Krasner's Lemma (see Lemma 8.1.6 of \cite{neukirch2008}) applied to the extension $K_{ur}\subset L$ that $I=\Gal(\bar{L}/L)$.

We define the Weil group $W$ of $K$ to be the pullback of $\frak{w}$ under the surjection $G\rightarrow \frak{g}$.  The diagram
$$
\xymatrix{
0\ar[r] &I\ar[r] &W\ar[r]\ar[d] &\frak{w}\ar[r]\ar[d] &0\\
0\ar[r] &I\ar[r]\ar@{=}[u] &G\ar[r] &\frak{g}\ar[r] &0
}
$$
best summarizes the relationship between all of these groups.  The topology on $W$ is such that $I$ is open, and translation by any preimage of $\sigma$ is a homeomorphism.  

By Chapter XIII, Lemma 1 of \cite{serre1979}, the fixed points of $\frak{w}$ acting on $L$ are exactly $K$; it is immediate that the fixed points of $W$ acting on $\bar{L}$ are also $K$.  Similar remarks apply to $\frak{w}$ and $W$ acting on the $L$ and $\bar{L}$-points, respectively, of some commutative algebraic group scheme defined over $K$.

The exact sequence $0\rightarrow I\rightarrow W\rightarrow \frak{w}\rightarrow 0$ of topological groups gives rise to a spectral sequence $H^{i}(\frak{w},H^{j}(I,M))\Rightarrow H^{i+j}(W,M)$ for any $W$-module $M$.  Because $\text{scd}(\frak{w})=1$, this spectral sequence degenerates to a collection of short exact sequences
\begin{equation}\label{weilses}
  0\rightarrow H^{1}(\frak{w},H^{i-1}(I,M))\rightarrow H^{i}(W,M)\rightarrow H^{0}(\frak{w},H^{i}(I,M))\rightarrow 0
\end{equation}
for all $i\geq 0$.   Recall from Chapter II \S 3.3 of \cite{serre2002} that $\text{cd}(I) = 1$ and $\text{scd}(I) = 2$; it follows that $H^i(W,M)=0$ for all $i\geq 4$ and all $M$; we therefore suppress any mention of $H^i(W,M)$ for $i\geq 4$ from now on.

\subsection{Cohomology of $\G_m$ and Finite $W$-Modules}

\begin{proposition}\label{weilecross}
The cohomology groups $H^i(W,\bar{L}^\times)$ are given by
$$H^{i}(W,\bar{L}^{\times}) = \left\{ \begin{array}{ll}
                                      K^{\times} & \textrm{if $i=0$}\\
                                      \Z & \text{if $i=1$}\\
                                      0 & \text{if $i\geq2$.}
                                      \end{array} \right.$$
\end{proposition}
\begin{proof} (See \cite{lichtenbaum1999}.)  The field $L$ is a $C_{1}$, hence  $H^{i}(I,\bar{L}^{\times})$ vanishes for $i\geq 1$.  It follows immediately from (\ref{weilses}) that $H^{i}(W,\bar{L}^{\times})=0$ for $i\geq2$.  For $i=1$, (\ref{weilses}) yields an isomorphism $H^{1}(\frak{w},L^{\times})\simeq H^{1}(W,\bar{L}^{\times})$.  The long exact sequence of $\frak{w}$-cohomology of $0\rightarrow U_{L}\rightarrow L^{\times}\rightarrow \Z\rightarrow 0$ now gives us a canonical isomorphism $H^{1}(W,\bar{L}^{\times})= H^{1}(\frak{w},\Z)=\Z$, since $H^1(\frak{w},U_L)=0$ by (\ref{H1vanish}).
\end{proof}

\begin{corollary}\label{weilmun}
  Let $\mu_{n}$ be the group of $n^{th}$ roots of unity in $\bar{L}^{\times}$, and let $\mu=\varinjlim_n \mu_n$.  Then
  $$H^{i}(W,\mu_{n}) = \left\{ \begin{array}{ll}
                                      \mu_{n}(K) & \textrm{if $i=0$}\\
                                      K^{\times}/(K^{\times})^{n} & \text{if $i=1$}\\
                                      \Z/n\Z & \text{if $i=2$}\\
                                      0      & \text{if $i\geq 3$,}
                                      \end{array} \right.$$
 and thus $H^2(W,\mu)=\varinjlim_n H^2(W,\mu_n)=\QZ$.
\end{corollary}
\begin{proof}
  Consider the long exact sequence in cohomology of the Kummer sequence $0\rightarrow \mu_{n}\rightarrow \bar{L}^{\times}\rightarrow \bar{L}^{\times}\rightarrow 0$.  Using (\ref{weilecross}), we see that the long exact sequence reads
  $$0\rightarrow \mu_{n}(K)\rightarrow K^{\times}\stackrel{n}{\rightarrow}K^{\times}\rightarrow H^{1}(W,\mu_{n})\rightarrow \Z\stackrel{n}{\rightarrow}\Z\rightarrow H^{2}(W,\mu_{n})\rightarrow 0$$
  from which the results follow immediately.
\end{proof}

In the following theorem we collect the Weil group analogues of classical results from Galois cohomology concerning finite modules.

\begin{theorem}\label{weiltld}
Let $M$ be a finite $W$-module.
\begin{itemize}
\item[(i)] The groups $H^i(W,M)$ are finite for all $i$, and vanish for $i\geq 3$.
\item[(ii)] We have $\text{cd}(W)=2$.
\item[(iii)] (Weil-Tate Local Duality) Let $M' = \Hom(M,\mu)$ with its natural $W$-action.  Then the cup-product pairing
$$
H^i(W,M)\otimes H^{2-i}(W,M')\rightarrow \QZ
$$
is a perfect pairing of finite groups.
\end{itemize}
\end{theorem}

\begin{proof}
The proof of (i) is the same as for the Galois group, see Chapter II, \S5.2 of \cite{serre2002}.  To prove (ii), recall that $\text{scd}(\frak{w})=1$ and $\text{cd}(N)=1$.  From the short exact sequences of \ref{weilses}, we see that $\text{cd}(W)\leq 2$.  But  $H^2(W,\mu_n)\neq 0$, hence the equality $\text{cd}(W)=2$.  The proof of (iii) is the same as for Tate Local Duality, as in Chapter II, \S 5.2 of \cite{serre2002}, provided one restricts to open subgroups of finite index of $W$ when using Shapiro's Lemma.
\end{proof}

\subsection{Vanishing of $H^3(W,M)$ for Finitely Generated $M$}

The vanishing theorem of this section is needed to prove the duality theorem of the next section, but it is interesting in its own right.  As a consequence of this vanishing theorem, we deduce a theorem of Rajan (and offer a slight correction to his proof).

\begin{theorem}\label{h3vanish}  Let $M$ be a $W$-module which is finitely generated as an abelian group.  Then $H^{3}(W,M)=0$.
\end{theorem}
\begin{proof}  The results of the previous section prove that $H^3(W,M)=0$ for finite $M$, so we may assume $M$ is free.  Let $\Q$ be the rationals with trivial $W$-action; then $M_{\Q}=M\otimes \Q$ and $P=M_{\Q}/M$ fit into the exact sequence $0\rightarrow M\rightarrow M_\Q\rightarrow P\rightarrow 0$ of $W$-modules.  The $\Q$-vector space $M_\Q$ is uniquely divisible, therefore $H^{1}(I,M_{\Q})=H^{2}(I,M_{\Q})=0$, which implies by (\ref{weilses}) that $H^{2}(W,M_{\Q})=H^{3}(W,M_{\Q})=0$.  Therefore $H^2(W,P)\simeq H^3(W,M)$, so it suffices to show this former group vanishes.

Let $P_{n}$ denote the $n$-torsion of $P$.  By Weil-Tate Local Duality (\ref{weiltld}) we have
$$
H^{2}(W,P) \simeq \varinjlim_{n} H^{2}(W,P_{n}) \simeq \varinjlim_{n} \Hom_{W}(P_{n},\mu)^* \simeq (\varprojlim_{n} \Hom_{W}(P_{n},\mu))^* \simeq \Hom_{W}(P,\mu)^*,
$$
where $\varprojlim_{n}\Hom_{W}(P_{n},\mu)$ has its profinite topology.  From these isomorphisms, we see that it suffices to show $\Hom_{W}(P,\mu)=0$.
\par Let $f:P\rightarrow \mu$ be $W$-equivariant, and let $I'$ be an open, normal subgroup of $I$ such that $I'$ acts trivially on $M$.  Then for all $\tau\in I'$ and $e\in P$, we have $f(e)=f(\tau\cdot e)=f(e)^{\tau}$.  It follows that $f(P)\cap \mu_{p^{\infty}}$ is finite, where $p=\text{char}(k)$.  Since $P$ has no finite quotients and we wish to prove that $f(P)=1$, we may assume that $f(P)\subseteq\mu'$, the group of roots of unity of order prime to $p$.  On $\mu'$, $W$ acts via the natural action of $\sigma$, which is $\zeta\mapsto \zeta^{q}$, where $q=\#k$.
\par Let $a\in W$ be a preimage of $\sigma$.  Applying the $W$-equivariance of $f$ to $a$, we see that $f(a\cdot e)=f(e)^{q}=f(q\cdot e)$ for all $e\in P$.  Therefore $(a - q)(P)\subseteq \ker(f)$.  We claim that $(a - q):P\rightarrow P$ is surjective, which implies $f=1$. By the Snake Lemma, it suffices to prove that $a - q$ is surjective on $M_{\Q}$.   Surjectivity of $a - q$ on $M_{\Q}$ is equivalent to injectivity, and injectivity will hold if and only if $q$ is not an eigenvalue of $a$.  But $q$ cannot be an eigenvalue of $a$, because $\det(a)=\pm 1$ and the characteristic polynomial of $a$ has integer coefficients.
\end{proof}

\begin{remark}
The above theorem does not immediately imply that $H^3(W,M)=0$ for all $W$-modules $M$, since one cannot write an arbitrary $W$-module as a direct limit of modules which are finitely generated as abelian groups.  For example, let $\frak{w}$ act on $M=\bigoplus_{i\in\Z}\Z$ by shifting the indices in the obvious way, and let $W$ act on $M$ through the quotient map $W\rightarrow \frak{w}$.
\end{remark}

Suppose now that $A$ is a \textit{topological} $W$-module; that is $A$ is a topological abelian group with a continuous action of $W$.  For such $A$, we can define cohomology groups $H^i_{cc}(W,A)$ using complexes of continuous cochains.  It follows from Corollary 2.4 of \cite{lichtenbaum2009} and Corollary 2 of \cite{flach2008} that the groups $H^i_{cc}(W,A)$ agree with the topological group cohomology used in \cite{lichtenbaum2009} and \cite{flach2008}, which is defined using the classifying topos $BW$.

The non-discrete topological $W$-modules in which we are interested are all complex manifolds, and hence Remark 2.2 of \cite{lichtenbaum2009} shows that the groups $H^i_{cc}(W,A)$ also agree with the cohomology groups $H^i_M(W,A)$ defined by Moore in \cite{moore3} and used by Rajan in \cite{rajan2004}.  If $A$ is discrete, then $H^i_{cc}(W,A)$ can be identified with the Galois cohomology groups $H^i(W,A)$; this follows from the remarks preceding Lemma 1 of \cite{rajan2004}.

\begin{theorem}\label{rajan}
Let $T$ be an algebraic torus over $\C$, equipped with an action of $W$ via algebraic automorphisms.  Then $H^{2}_{cc}(W,T(\C))=0$.  In particular, for $T=\G_m$, we conclude that $H^2(W,\C^\times)=0$.
\end{theorem}

\begin{proof} From Corollary 8 of \cite{flach2008}, it follows that $H^{i}_{cc}(I,V)=0$ for any finite-dimensional complex vector space $V$ and any $i\geq 1$.  Thus the spectral sequence of Moore (quoted by Rajan as Proposition 5 of \cite{rajan2004}) gives isomorphisms $H^i(\frak{w},H^0_{cc}(I,V))= H^i_{cc}(W,V),$ and thus this latter group vanishes for $i\geq 2$.

Let $X_*(T)$ denote the cocharacter group of $T$; it is a finitely generated discrete $W$-module.  There is a short exact sequence of topological $W$-modules,
$$0\rightarrow X_*(T)\rightarrow X_*(T)\otimes\C\rightarrow T(\C)\rightarrow 0.$$
The groups $H^i_{cc}(W,X_*(T)\otimes\C)$ vanish for $i\geq 2$ by the previous paragraph, and the groups $H^i(W,X_*(T))$ vanish for $i\geq 3$ by (\ref{h3vanish}).  The theorem now follows by considering the long exact  sequence in cohomology.
\end{proof}

In \cite{rajan2004}, Rajan proves that $H^{2}_M(W,T(\C))=0$, for those tori with an action of $W$ that comes from an action of $G$.  His proof seems to be slightly flawed, because his Proposition 6 asserts that $H^2(G,A)\rightarrow H^2(W,A)$ is an isomorphism for any $G$-module $A$, when in fact this map has a kernel for $A=\Z$.  However, Rajan's proof of the vanishing of $H^{2}_M(W,T(\C))$ ultimately relies only on the surjectivity of $H^{2}(G,A)\rightarrow H^{2}(W,A)$, which holds by by (\ref{weilgalcomp}).

Theorem (\ref{rajan}) implies that the map $H^1(W,\text{GL}_n(\C))\rightarrow H^1(W,\text{PGL}_n(\C))$ is surjective, which says exactly that every projective complex representation of $W$ lifts to an affine representation.

\subsection{Duality for Finitely Generated Modules}

Let $\Gamma_W:W\text{-Mod}\rightarrow \mathcal{A}b$ be the fixed-points functor, which has derived functor $R\Gamma_W:\mathcal{D}(W)\rightarrow \mathcal{D}(\Z)$, so that $H^i(R\Gamma_W(M))=H^i(W,M)$ for any $W$-module $M$.  By (\ref{weilecross}), there is a natural projection map $R\Gamma_W(\bar{L}^\times)\rightarrow \Z[-1]$.  If $M$ is any bounded complex of $W$-modules, we set $M^D:= R\Hom(M,\bar{L}^\times)$.  The cup-product pairing
$$
R\Gamma_W(M)\otimes^L R\Gamma_W(M^D)\rightarrow \Z[-1]
$$
induces a map
\begin{equation}\label{inducedbycup}
\psi(M):R\Gamma_W(M^D)\rightarrow R\Hom(R\Gamma_W(M),\Z[-1]).
\end{equation}

\begin{theorem}\label{fgweilduality}
Suppose that $M$ is a bounded complex of $W$-modules, whose cohomology groups are finitely generated as abelian groups.  Then the map $\psi(M)$ of (\ref{inducedbycup}) is an isomorphism in $\mathcal{D}(\Z)$.  
\end{theorem}

By induction on the length of the complex, it is clear that it suffices to prove the result for $M$ concentrated in degree $0$.  Furthermore, for finite $M$, (\ref{fgweilduality}) is easily seen to be equivalent to (\ref{weiltld}).  So it suffices to prove (\ref{fgweilduality}) for finitely generated, free $W$-modules $M$.   For the rest of the section, we let $M$ be such a module.

The proof of this theorem relies on several lemmas.  To begin, the discrete valuation $L^\times\rightarrow \Z$ induces a valuation $\bar{L}^\times\rightarrow \Q$.  Let $U_{\bar{L}}$ be the kernel of this map, so that we have a short exact sequence
$$
0\rightarrow U_{\bar{L}}\rightarrow \bar{L}^\times\rightarrow \Q\rightarrow 0
$$
of $W$-modules.  Taking $W$-cohomology, it follows easily that $H^2(W,U_{\bar{L}})=\QZ$, and $H^i(W,U_{\bar{L}})=0$ for $i\geq 3$.  Cup-product therefore induces the map
\begin{equation}\label{fivelemma}
\xymatrix{
R\Gamma_W(\Hom(M,U_{\bar{L}}))\ar[r]\ar[d] &R\Hom(R\Gamma_W(M),\QZ[-2])\ar[d] \\
R\Gamma_W(M^D)\ar[d]\ar[r]^-{\psi(M)} &R\Hom(R\Gamma_W(M),\Z[-1])\ar[d] \\
R\Gamma_W(\Hom(M,\Q))\ar[r] &R\Hom(R\Gamma_W(M),\Q[-1])
}
\end{equation}
 of exact triangles in $\mathcal{D}(\Z)$.

\begin{lemma}
To prove Theorem (\ref{fgweilduality}), it suffices to show that the maps
$$
H^i(W,\Hom(M,U_{\bar{L}}))\rightarrow H^{2-i}(W,M)^*
$$
and
$$
H^j(W,\Hom(M,\Q))\rightarrow \Hom(H^{1-j}(W,M),\Q)
$$
induced by cup-product are isomorphisms for all $i\geq 0$ and all $j\geq 0$.
\end{lemma}
\begin{proof}
By the Five Lemma, it to prove Theorem (\ref{fgweilduality}), it suffices to prove that the top and bottom maps of (\ref{fivelemma}) are isomorphisms.  Because $\Q$ and $\QZ$ are injective abelian groups, this is easily seen to be equivalent to the statement of the lemma.
\end{proof}


\begin{proposition}\label{weilqduality}
The map
$$H^j(W,\Hom(M,\Q))\rightarrow \Hom(H^{1-j}(W,M),\Q)$$
induced by cup-product is an isomorphism for all $j\geq 0$.
\end{proposition}

\begin{proof}
Because the quotient $M_\Q/M$ has only trivial maps to $\Q$, our map can be identified with
$$H^j(W,\Hom_\Q(M_\Q,\Q))\rightarrow \Hom_\Q(H^{1-j}(W,M_\Q),\Q).$$
Since $H^i(I,V)$ vanishes for $i\geq 1$, proving the proposition is equivalent, by (\ref{weilses}), to showing that the map 
$$H^j(\frak{w},\Hom_\Q(M_\Q^I,\Q))\rightarrow \Hom_\Q(H^{1-j}(\frak{w},M_\Q^I),\Q)$$
is an isomorphism for all $j\geq 0$.  This follows immediately from our results on $\frak{w}$-duality for vector space coefficients.
\end{proof}

To prove that the maps $H^i(W,\Hom(M,U_{\bar{L}}))\rightarrow H^{2-i}(W,M)^*$ are all isomorphisms, we will use (\ref{weilses}).  In particular, for every $i\geq 0$, we have a map of short exact sequences
\begin{equation}\label{fivelemma2}
\xymatrix{
0\ar[d] &0\ar[d]\\
H^1(\frak{w},H^{i-1}(I,\Hom(M,U_{\bar{L}})))\ar[r]\ar[d] &H^0(\frak{w},H^{2-i}(I,M))^*\ar[d]\\
H^i(W,\Hom(M,U_{\bar{L}}))\ar[r]\ar[d] &H^{2-i}(W,M)^*\ar[d]\\
H^0(\frak{w},H^i(I,\Hom(M,U_{\bar{L}}))\ar[r]\ar[d] &H^1(\frak{w},H^{1-i}(I,M))^*\ar[d]\\
0 &0
}
\end{equation}
and to prove that the middle arrow is an isomorphism, we will prove, for $i\geq 1$, that the top and bottom horizontal arrows are isomorphisms.  It should be noted that the top and bottom arrows are not isomorphisms for all $i\geq 0$: take $M=\Z$ and $i=0$, then the bottom arrow is the zero map $U_K\rightarrow 0$.

Let $L'/L$ be a finite Galois extension with group $H$ and degree $e$.  Consider the long exact sequence in $H$-cohomology of $0\rightarrow U_{L'}\rightarrow L'^{\times}\rightarrow \Z\rightarrow 0$.  Since the groups $H^i(H,L'^{\times})$ and $H^i(H,\Q)$ vanish for $i\geq 1$, the long exact sequence reads
$$
0\rightarrow U_L\rightarrow L^\times\rightarrow \Z\rightarrow H^1(H,U_{L'})\rightarrow 0
$$
which gives us a canonical identification $H^1(H,U_{L'})=\Z/e\Z$.  Taking the limit over all $L'/L$ gives an identification $H^1(I,U_{\bar{L}})=\QZ$.

Throughout the proof of the next proposition, we will make use of the Tate cohomology groups $\hat{H}^*(H,-)$, for a finite group $H$.  For definitions and basic properties of Tate cohomology groups, see Chapter VIII of \cite{serre1979}.

\begin{proposition}\label{intermediateduality}
The maps
$$
H^i(\frak{w},H^1(I,\Hom(M,U_{\bar{L}})))\rightarrow H^{1-i}(\frak{w},H^0(I,M))^*
$$
are isomorphisms for all $i\geq 0$.
\end{proposition}
\begin{proof}

Let $M$ be an $I$-module which is finitely generated and free as an abelian group.  We will show that
$$
\alpha: H^1(I,\Hom(M,U_{\bar{L}})))\rightarrow H^{0}(I,M)^*
$$
is an isomorphism.  The proposition then follows easily from $\frak{w}$-duality.  If $M=\Z$ with trivial $I$-action, then $\alpha$ is the canonical map $\QZ\rightarrow \Z^*$.  Since cohomology commutes with direct sums, this proves the result for $M$ with trivial action.

Now let $M$ be any finitely generated $I$-module, and choose a finite Galois extension $L'/L$ with group $H$, such that the open subgroup $I'=\Gal(\bar{L}/L')$ of $N$ acts trivially on $M$.   We mimic the argument on page 128 of \cite{lichtenbaum1969}.  Consider the diagrams
$$
\xymatrix{
0\ar[r] &H^{1}(H,\Hom(M,U_{L'}))\ar[r]^-{\text{inf}}\ar[d]^-{\gamma} &H^{1}(I,\Hom(M,U_{\bar{L}}))\ar[r]^-{\text{res}}\ar[d]^-{\alpha} &H^{1}(I',\Hom(M,U_{\bar{L}}))\ar[d]^-{\alpha'}\\
0\ar[r] &\hat{H}^{0}(H,M)^{*}\ar[r] &H^{0}(I,M)^{*}\ar[r]^-{\text{tr}^{*}} &M^{*}
}
$$
and
$$
\xymatrix{
H^{1}(I',\Hom(M,U_{\bar{L}}))\ar[r]^-{\text{tr}}\ar[d]^-{\alpha'} &H^{1}(I,\Hom(M,U_{\bar{L}}))\ar[d]^-{\alpha}\\
M^{*}\ar[r] &H^{0}(I,M)^{*}\ar[r] &0
}
$$
Where $\alpha'$ is the corresponding map for $L'$.  By the previous paragraph, $\alpha'$ is an isomorphism.  From the second of the above two diagrams, we see that $\alpha$ is surjective.  To show that $\alpha$ is an isomorphism, we only need to show it is injective, and for this, it suffices to prove that $\gamma$ is an isomorphism.

By Chapter X, \S 7, Proposition 11 and Chapter IX, \S 5, Theorem 9 of \cite{serre1979}, $\Hom(M,L'^{\times})$ is cohomologically trivial for $H$.  Applying the exact functor $\Hom(M,-)$ to $0\rightarrow U_{L'}\rightarrow L'^{\times}\rightarrow \Z\rightarrow 0$ and taking reduced $H$-cohomology gives us an isomorphism
$$
\delta:\hat{H}^0(H,\Hom(M,\Z))\rightarrow H^1(H,\Hom(M,U_{L'}))
$$
which commutes with cup-product in the sense that the diagram
$$
\xymatrix{
\hat{H}^0(H,\Hom(M,\Z))\ar[r]\ar[d]^-\delta &\hat{H}^0(H,M)^*\ar@{=}[d]\\
H^1(H,\Hom(M,U_{L'}))\ar[r]^-\gamma &\hat{H}^0(H,M)^*
}
$$
commutes.  By (\cite{neukirch2008}, Chapter III, Proposition 3.1.2) the top horizontal arrow is an isomorphism, and we conclude that $\gamma$ is an isomorphism.
\end{proof}

\begin{proposition}\label{intermediateduality2}
The map
$$
H^1(\frak{w},H^0(I,\Hom(M,U_{\bar{L}})))\rightarrow H^0(\frak{w},H^1(I,M))^*
$$
induced by cup-product is an isomorphism.
\end{proposition}

\begin{proof}
Let $L'/L$ be a finite Galois extension with group $H$, such that $I'=\Gal(\bar{L}/L')$ acts trivially on $M$.  Using the same argument as in the previous proposition and (\cite{neukirch2008}, Chapter III, Proposition 3.1.2), we see that the map
$$
\hat{H}^0(H,\Hom(M,U_{L'}))\rightarrow H^1(H,M)^*
$$
is an isomorphism.  

We claim that there is an isomorphism 
$$
H^1(\frak{w},H^0(N,\Hom(M,U_{\bar{L}})))=H^1(\frak{w},\hat{H}^0(H,\Hom(M,U_{L'}))).
$$
Together with the fact that the inflation map $H^1(H,M)\rightarrow H^1(I,M)$ is an isomorphism (since $H^1(I',M)=0$), this will suffice to prove the proposition.  It is clear that $H^0(I,\Hom(M,U_{\bar{L}}))=H^0(H,\Hom(M,U_{L'}))$, which reduces us to showing that the natural map
$$
H^1(\frak{w},H^0(H,\Hom(M,U_{L'})))\rightarrow H^1(\frak{w},\hat{H}^0(H,\Hom(M,U_{L'})))
$$
is an isomorphism.

For any $H$-module $P$, let $\text{Nm}_H:P\rightarrow H^0(H,P)$ be the norm map, defined by $m\mapsto \sum_{s\in H}s\cdot m$.  Consider the reduction map 
$$\text{Nm}_{H}(\Hom(M,U_{L'}))\rightarrow \text{Nm}_{H}(\Hom(M,U_{L'}/U_{L'}^{(1)}))=\text{Nm}_{H}(\Hom(M,\bar{k}^{\times})).$$
The extension $L'/L$ is totally ramified, hence $H$ acts trivially on $\bar{k}^{\times}$.  Thus for any $f:M\rightarrow \bar{k}^{\times}$, we have
$$\text{Nm}_{H}(f)(m) = \sum_{s\in H}sf(s^{-1}m) = \sum_{s\in H}f(s^{-1}m) = f(\text{Nm}_{H}(m))$$
which proves that $\text{Nm}_{H}(\Hom(M,\bar{k}^{\times}))=\Hom(\text{Nm}_{H}(M),\bar{k}^{\times})$.  As $\text{Nm}_{H}(M)\subset M$ is a finitely generated and free abelian group, we see from (\ref{H1vanish}) that $H^1(\frak{w},\Hom(\text{Nm}_{H}(M),\bar{k}^{\times}))=0$, and the same argument shows that $H^1(\frak{w},\Hom(\text{Nm}_{H}(M),\bar{k}))=0$. One now proceeds in the same fashion as in the proof of (\ref{H1vanish}) to show that $H^{1}(\frak{w},\text{Nm}_{H}(\Hom(M,U_{L'})))=0$.  
\end{proof}

\begin{proposition}
The maps
\begin{equation}\label{wcup}
H^i(W,\Hom(M,U_{\bar{L}}))\rightarrow H^{2-i}(W,M)^*
\end{equation}
are isomorphisms for all $i\geq 0$.  For $i=0$, this is a topological isomorphism of profinite groups.
\end{proposition}

\begin{proof}
For $i\geq 1$, this follows immediately from (\ref{fivelemma2}), (\ref{intermediateduality}), and (\ref{intermediateduality2}).  For $i=0$, multiplication-by-$n$ Kummer sequences for $M$ and $\Hom(M,U_{\bar{L}})$ give rise to a map of long exact cohomology sequences, the relevant part of which reads
$$
\xymatrix{
0\ar[r] &H^{0}(\Hom(M,U_{\bar{L}}))/n\ar[r]^-{\delta}\ar[d] &H^{1}(\Hom(M,U_{\bar{L}})[n])\ar[r]\ar[d]^-\wr &H^{1}(\Hom(M,U_{\bar{L}}))[n]\ar[d]^-\wr\\
0\ar[r] &H^{2}(M)[n]^{*}\ar[r]^-{\delta^{*}} &H^{1}(M/n)^{*}\ar[r] &(H^{1}(M)/n)^{*}
}
$$
where we have written $H^i(-)$ for $H^i(W,-)$ for convenience.  The middle arrow is Weil-Tate Local Duality applied to the finite modules $\Hom(M,U_{\bar{L}})[n]=\Hom(M/n,\mu_n)$ and $M/n$.  The right-most arrow is the map of (\ref{wcup}) when $i=1$, restricted to $n$-torsion; it is therefore an isomorphism.  We conclude by the Five Lemma that the left-most vertical arrow is an isomorphism for all $n$.  Taking the inverse limit, we conclude that
$$
\varprojlim_n H^{0}(W,\Hom(M,U_{\bar{L}}))/n \rightarrow H^2(W,M)^*
$$
is an isomorphism (recall that $H^2(W,M)$ is torsion).

It remains to show that the natural map 
$$
H^{0}(W,\Hom(M,U_{\bar{L}}))\rightarrow \varprojlim_{n}H^{0}(W,\Hom(M,U_{\bar{L}}))/n
$$
is an isomorphism; in other words, to show that $H^{0}(W,\Hom(M,U_{\bar{L}}))$ is profinite.  If $I$ acts trivially on $M$, then this is contained in the statement of (\ref{H0finite}).  Otherwise, pick an open normal subgroup $I'$ of $I$ acting trivially on $M$, corresponding to a finite Galois extension $L'/L$ with group $H$.  We have
$$H^{0}(W,\Hom(M,U_{\bar{L}})) = H^{0}(\frak{w},H^{0}(H,\Hom(M,U_{L'})))$$
and the latter group is clearly profinite, hence we are done.
\end{proof}

Theorem (\ref{fgweilduality}) was proven by Jiang in his thesis (see Proposition 4.15 and Theorem 5.3 of \cite{jiang2006}).  However, Jiang assumes that $M$ is a $G$-module, and uses Tate-Nakayama Duality in his proof.  We have removed the condition that $M$ be a $G$-module, and presented a proof which is independent of the main results of Local Class Field Theory.

A natural question to ask is whether the other map induced by the cup-product pairing, namely
\begin{equation}
\eta(M):R\Gamma_W(M)\rightarrow R\Hom(R\Gamma_W(M^D),\Z[-1]),
\end{equation}
is an isomorphism.  The next proposition shows that this map fails to be an isomorphism when $M=\Z$ with trivial action, due to the non-trivial natural topology on the cohomology groups of the complex $R\Gamma_W(M^D)$.  However, elucidating this map will prove useful in later sections, so we now describe it explicitly.

\begin{proposition}\label{almostduality}
The map
\begin{equation}
\eta(\Z):R\Gamma_W(\Z)\rightarrow R\Hom(R\Gamma_W(\bar{L}^\times),\Z[-1])
\end{equation}
has the following properties:
\begin{itemize}
\item[(i)] $\eta(\Z)^i$ is an isomorphism for $i\neq 2$.
\item[(ii)] $\eta(\Z)^2$ induces an isomorphism of $H^2(W,\Z)$ with the torsion subgroup $U_K^*$ of $\Ext^2(R\Gamma_W(\bar{L}^\times),\Z[-1])$.
\end{itemize}
The cohomology of both complexes vanishes outside of degrees $0$ through $2$.
\end{proposition}

\begin{proof}
For $i=0,1$, explicit calculation using (\ref{rhomz}) shows that the cohomology of both sides is $\Z$, and the map between them is the identity map.  For $i\geq 3$, using (\ref{rhomz}) one shows easily that the cohomology of both complexes vanishes.

The only assertion left to prove is that of (ii).  Coming from the sequences $0\rightarrow \Z\rightarrow \Z\rightarrow \Z/n\Z\rightarrow 0$ and $0\rightarrow \mu_n\rightarrow U_{\bar{L}}\rightarrow U_{\bar{L}}\rightarrow 0$ we have a map of short exact sequences
$$
\xymatrix{
0\ar[r] &H^1(W,\Z)/n\ar[r]\ar[d]^-\wr &H^1(W,\Z/n\Z)\ar[r]^\delta\ar[d]^-\wr &H^2(W,\Z)[n]\ar[r]\ar[d] &0\\
0\ar[r] &(H^1(W,U_{\bar{L}})^*)/n\ar[r] &H^1(W,\mu_n)^*\ar[r]^-{\delta^*} &U_K^*[n]\ar[r] &0.
}
$$
Here the left vertical arrow is the natural isomorphism $\Z/n\Z\rightarrow \hat{\Z}/n\hat{\Z}$, and the middle arrow is the isomorphism of Weil-Tate Local Duality.  Therefore the right vertical arrow is an isomorphism, and by passing to the limit over all $n$, we see that $H^2(W,\Z)\rightarrow U_K^*$ is an isomorphism.  A simple calculation using (\ref{rhomz}) shows that $\Ext^2(R\Gamma_W(\bar{L}^\times),\Z[-1])=\Ext(U_K,\Z)$, whose torsion subgroup is $U_K^*$.
\end{proof}

\begin{corollary}
The cohomology groups $H^i(W,\Z)$ are given by
$$
H^i(W,\Z)=\left\{
\begin{array}{ll}
\Z & i=0,1\\
U_K^* & i=2\\
0 & i\geq 3.
\end{array}
\right.
$$
\end{corollary}

\begin{proof}
This is contained in the proof of the previous proposition.
\end{proof}


\section{Local Class Field Theory via the Weil Group}
\label{sec:Local Class Field Theory via the Weil Group}

In this section we show how to deduce the main theorems of Local Class Field Theory from the theorems of the previous chapter.  We prove, in particular, that $\text{Br}(K)=\QZ$ and that there is a canonical isomorphism $K^\times\otimes\hat{\Z}\rightarrow G^{ab}$.

\subsection{Comparison with Galois Cohomology}

Any $G$-module can be given a $W$-module structure via the map $W\rightarrow G$, and this process induces restriction maps $H^i(G,M)\rightarrow H^i(W,M)$ on cohomology.  The following comparison theorems describe these restriction maps.

\begin{proposition}\label{weilgaltorsion}
Let $M$ be a torsion $G$-module.  Then there are functorial isomorphisms $H^i(G,M)= H^i(W,M)$ for all $i\geq 0$.
\end{proposition}

\begin{proof} (This proof appears in the unpublished note \cite{lichtenbaum1999} of Lichtenbaum.)  Since Galois cohomology is always torsion, and $\frak{g}$ and $\frak{w}$ cohomology agree for torsion coefficients by Chapter XIII, Proposition 1 of \cite{serre1979}, the restriction maps $H^i(\frak{g},H^j(N,M))\rightarrow H^i(\frak{w},H^j(N,M))$ are isomorphisms for all $i,j\geq 0$.  Thus the map of spectral sequences
$$
\xymatrix{
H^i(\frak{g},H^j(N,M))\ar@{=>}[r]\ar[d] &H^{i+j}(G,M)\ar[d]\\
H^i(\frak{w},H^j(N,M))\ar@{=>}[r] &H^{i+j}(W,M)
}
$$
is an isomorphism on the second page and therefore in the limit.
\end{proof}

\begin{corollary} (Tate Local Duality)
Let $M$ be a finite $G$-module.  The cup-product pairing 
$$
H^i(G,M)\otimes H^{2-i}(G,M')\rightarrow \QZ
$$
is a perfect pairing of finite groups.
\end{corollary}

\begin{proof}
This follows from (\ref{weiltld}), and (\ref{weilgaltorsion}) applied to the finite module $M$.
\end{proof}

\begin{theorem}\label{weilgalcomp}
Let $M$ be a discrete $G$-module.  Then there are functorial isomorphisms
\begin{itemize}

\item[(i)] $H^0(G,M)= H^0(W,M)$ and

\item[(ii)] $H^1(G,M)= H^1(W,M)_{\text{tors}},$

\item[(iii)] there is a short exact sequence
$$0\rightarrow H^1(W,M)\otimes \QZ\rightarrow H^2(G,M)\rightarrow H^2(W,M)\rightarrow 0,$$

\item[(iv)] and there are isomorphisms $H^i(W,M)=H^i(G,M)$ for all $i\geq 3$.
\end{itemize}
\end{theorem}

\begin{proof}
By considering the exact sequence $0\rightarrow M_{\text{tors}}\rightarrow M\rightarrow M/M_{\text{tors}}\rightarrow 0$ and using the above proposition, we may assume $M$ is torsion-free.  Since $W$ is dense in $G$ it is clear that $H^0(G,M)= H^0(W,M)$.  Kummer sequences give rise to a diagram
$$
\xymatrix{
0\ar[r] &H^0(G,M)/n\ar[r]\ar[d]^-\wr &H^0(G,M/nM)\ar[r]^-\delta\ar[d]^-\wr &H^1(G,M)[n]\ar[r]\ar[d] &0\\
0\ar[r] &H^0(W,M)/n\ar[r] &H^0(W,M/nM)\ar[r]^-\delta &H^1(W,M)[n]\ar[r] &0.
}
$$
Since $H^1(G,M)$ is all torsion, passing to the limit proves that $H^1(G,M)= H^1(W,M)_{\text{tors}}$.

For part (iii), consider the diagram
$$
\xymatrix{
0\ar[r] &H^1(G,M)/n\ar[r]\ar[d]^-\beta &H^1(G,M/n)\ar[r]^-\delta\ar[d]^-\wr &H^2(G,M)[n]\ar[d]^-\kappa\ar[r] &0\\
0\ar[r] &H^1(W,M)/n\ar[r] &H^1(W,M/n)\ar[r]^-\delta &H^2(W,M)[n]\ar[r] &0.
}
$$
From the Snake Lemma it is clear that $\kappa$ is surjective and that $\ker(\kappa)=\text{coker}(\beta)$.  Passing to the limit over all $n$, we have
\begin{eqnarray*}
\varinjlim_n\text{coker}(H^1(G,M)/n\rightarrow H^1(W,M)/n) &=& \text{coker}(H^1(G,M)\otimes\QZ\rightarrow H^1(W,M)\otimes\QZ)\\
 &=& H^1(W,M)\otimes\QZ
\end{eqnarray*}
because $H^1(G,M)\otimes\QZ=0$.  The existence of our exact sequence is now clear.

For $i= 3$, we have the diagram
$$
\xymatrix{
0\ar[r] &H^2(G,M)/n\ar[r]\ar[d] &H^2(G,M/n)\ar[r]^-\delta\ar[d]^-\wr &H^3(G,M)[n]\ar[d]\ar[r] &0\\
0\ar[r] &H^2(W,M)/n\ar[r] &H^2(W,M/n)\ar[r]^-\delta &H^3(W,M)[n]\ar[r] &0.
}
$$
and it follows from the exact sequence $0\rightarrow H^1(W,M)\otimes \QZ\rightarrow H^2(G,M)\rightarrow H^2(W,M)\rightarrow 0$ that the left arrow is an isomorphism.  Therefore $H^3(G,M)[n]=H^3(W,M)[n]$, and passing to the limit over all $n$ gives $H^3(G,M)=H^3(W,M)$.  The result for $i\geq 4$ follows immediately by induction on $i$ and considering Kummer sequences.
\end{proof}

The above theorem implies that the groups $H^i(W,M)$ determine the groups $H^i(G,M)$ up to isomorphism, since $H^1(W,M)\otimes\QZ$ is an injective abelian group.  The converse fails: take for example $M=\Q$ with trivial action.  Then $H^i(G,\Q)=0$ for $i\geq 1$, but $H^1(W,\Q)=\Hom(\frak{w},\Q)=\Q$.  Thus the groups $H^i(W,M)$ contain more information than their Galois counterparts.

\begin{corollary}
The strict cohomological dimension of $G$ is $2$.
\end{corollary}
\begin{proof}
Let $M$ be a $G$-module.  Since the orbit under $G$ of any $m\in M$ is finite, we can write $M$ as a direct limit of $G$-modules $M_n$ which are finitely generated as abelian groups.  Then by the above comparison theorem and (\ref{h3vanish}), we have $H^3(G,M)=\varinjlim_n H^3(W,M_n)=0$
\end{proof}

Let $\mathcal{A}/K$ be a commutative algebraic group scheme defined over $K$.  Then $\mathcal{A}(\bar{K})$ is naturally a $G$-module, and $\mathcal{A}(\bar{L})$ is naturally a $W$-module.  The inclusion map $\mathcal{A}(\bar{K})\rightarrow \mathcal{A}(\bar{L})$ induces restriction maps $H^i(G,\mathcal{A}(\bar{K}))\rightarrow H^i(W,\mathcal{A}(\bar{L}))$.  We have the following comparison theorem:

\begin{theorem}\label{weilgalalgcomp}
Let $\mathcal{A}$ be a connected commutative algebraic group scheme over $K$.  There are functorial isomorphisms:

\begin{itemize}

\item[(i)] $H^0(G,\mathcal{A}(\bar{K}))= H^0(W,\mathcal{A}(\bar{L}))$ and

\item[(ii)] $H^1(G,\mathcal{A}(\bar{K}))= H^1(W,\mathcal{A}(\bar{L}))_{\text{tors}},$

\item[(iii)] there is a short exact sequence
$$0\rightarrow H^1(W,\mathcal{A}(\bar{L}))\otimes \QZ\rightarrow H^2(G,\mathcal{A}(\bar{K}))\rightarrow H^2(W,\mathcal{A}(\bar{L}))\rightarrow 0,$$

\item[(iv)] and the higher cohomology groups $H^i(G,\mathcal{A}(\bar{K}))$ and $H^i(W,\mathcal{A}(\bar{L}))$ vanish for $i\geq 3$.
\end{itemize}
\end{theorem}

\begin{proof}
First we will show that the multiplication-by-$n$ maps $\mathcal{A}\rightarrow \mathcal{A}$ are all surjective.  Since $\mathcal{A}$ is connected, we have a short exact sequence
$$0\rightarrow H\rightarrow \mathcal{A}\rightarrow B\rightarrow 0$$
of algebraic groups, where $H$ is linear and $B$ is an abelian variety (see Theorem 1.1 of \cite{conrad2002}).  As $H$ is commutative, it is the product of a torus and a commutative unipotent group.  It follows that the multiplication-by-$n$ maps on $H$ and $B$ are surjective, hence the same is true of $\mathcal{A}$ by the Five Lemma.

The $n$-torsion of $\mathcal{A}(\bar{L})$ is contained in $\mathcal{A}(\bar{K})$, hence there is a map of Kummer sequences
$$
\xymatrix{
0\ar[r] &\mathcal{A}[n]\ar[r]\ar@{=}[d] &\mathcal{A}(\bar{K})\ar[r]^-n\ar[d] &\mathcal{A}(\bar{K})\ar[r]\ar[d] &0\\
0\ar[r] &\mathcal{A}[n]\ar[r] &\mathcal{A}(\bar{L})\ar[r]^-n &\mathcal{A}(\bar{L})\ar[r] &0.
}
$$
The rest of the proof is exactly as in (\ref{weilgalcomp}).
\end{proof}

\begin{corollary}
There is a natural isomorphism $H^1(W,\bar{L}^\times)\otimes \QZ= \text{Br}(K)$, and thus $\text{Br}(K)=\QZ$.
\end{corollary}

\begin{proof}
The isomorphism is the map in the short exact sequence of (\ref{weilgalalgcomp}), and the second statement follows from (\ref{weilecross}).
\end{proof}

From now on, we denote the groups $H^i(G,\mathcal{A}(\bar{K}))$ and $H^i(W,\mathcal{A}(\bar{L}))$ by $H^i(G,\mathcal{A})$ and $H^i(W,\mathcal{A})$, respectively.

\subsection{The Reciprocity Isomorphism}
Recall the main theorem (\ref{fgweilduality}) of the previous section, which states that for a $W$-module $M$ which is finitely generated as an abelian group, cup-product gives a natural isomorphism $R\Gamma_W(M^D)\stackrel{\sim}{\rightarrow} R\Hom(R\Gamma_W(M),\Z[-1])$ in the derived category of abelian groups.  Suppose now that $T/K$ is a torus with character group $M=\underline{\Hom}(T,\G_m)$.  Then $T(\bar{L})$ can be identified with $M^D$ as a $W$-module, and the duality theorem reads
$$
R\Gamma_W(T)\stackrel{\sim}{\rightarrow} R\Hom(R\Gamma_W(M),\Z[-1]).
$$
We can use our duality theorem to prove the following:

\begin{proposition}\label{h2torusvanish}
Let $T/K$ be a torus.  Then $H^2(W,T)=0$.
\end{proposition}

\begin{proof}
By the duality theorem, this group is isomorphic to $\Ext^2(R\Gamma_W(M),\Z[-1])$.  By (\ref{rhomz}), there is an isomorphism $\Ext(H^0(W,M),\Z)\stackrel{\sim}{\rightarrow} \Ext^2(R\Gamma_W(M),\Z[-1]),$ but the former group vanishes because $H^0(W,M)\subseteq M$ is finitely generated and free.
\end{proof}

\begin{corollary}\label{weilgaltorcomp}
Let $T/K$ be a torus.  Then there are natural isomorphisms
\begin{itemize}
\item[(i)] $H^0(G,T)=H^0(W,T)$
\item[(ii)] $H^1(G,T)=H^1(W,T)_{\text{tors}}$
\item[(iii)] $H^2(G,T)= H^1(W,T)\otimes\QZ$
\end{itemize}
\end{corollary}

\begin{proof}
This follows immediately from (\ref{weilgalalgcomp}) and (\ref{h2torusvanish}).
\end{proof}

\begin{corollary} (Tate-Nakayama Duality) 
Let $T/K$ be a torus with character group $M$.  Then the map
$$
H^i(G,T)\rightarrow H^{2-i}(G,M)^*
$$
induced by cup-product is an isomorphism for $i=1,2$, and an isomorphism for $i=0$ upon passing to the profinite completion of the left-hand side.
\end{corollary}

\begin{proof}
Using (\ref{weilses}) one sees that $H^1(W,M)$ is finitely generated, hence $\Ext(H^1(W,M),\Z)\otimes\QZ$ vanishes.  Our duality theorem therefore gives an isomorphism
$$
H^1(W,T)\otimes\QZ\stackrel{\sim}{\rightarrow}\Hom(H^0(W,M),\Z)\otimes\QZ=H^0(W,M)^*
$$
which, by the previous corollary, can be identified with the map $H^2(G,T)\rightarrow H^0(G,M)^*$.

Now consider the case where $i=1$.  Our duality theorem gives an isomorphism 
$$
H^1(W,T)_{\text{tors}}\stackrel{\sim}{\rightarrow}\Ext(H^1(W,M),\Z)=(H^1(W,M)_{\text{tors}})^*
$$
which, again by the previous corollary, can be identified with the map $H^1(G,T)\rightarrow H^1(G,M)^*$.

Finally we treat the case $i=0$.  Consider the commutative diagram
$$
\xymatrix{
0\ar[r] &\Ext(H^2(W,M),\Z)\ar[r]\ar[d]^-= &H^0(G,T)\ar[r]\ar[d] &\Hom(H^1(W,M),\Z)\ar[r]\ar[d] &0\\
0\ar[r] &H^2(W,M)^*\ar[r] &H^2(G,M)^*\ar[r] &(H^1(W,M)\otimes\QZ)^*\ar[r] &0
}
$$
where the top row comes from the identity $H^0(G,T)=H^0(W,T)$, our duality theorem, and (\ref{rhomz}).  The bottom row is the dual of the exact sequence of (\ref{weilgalcomp}).  The left vertical arrow is the identity since $H^2(W,M)$ is all torsion, as can be seen easily from (\ref{weilses}).  The right vertical arrow is an isomorphism upon passing to the profinite completion of $\Hom(H^1(W,M),\Z)$, hence the same is true of the middle vertical arrow.
\end{proof}

Of course, for $T=\G_m$ and $i=0$, one recovers from Tate-Nakayama Duality the reciprocity isomorphism $K^\times\otimes\hat{\Z}\stackrel{\sim}{\rightarrow} G^{ab}$ of Local Class Field Theory, where $-\otimes \hat{\Z}$ denotes profinite completion.  Hence one can recover the main statements of Local Class Field Theory by studying the cohomology of the Weil group.


\section{The Weil-smooth Topology on Schemes over $K$}
\label{sec:The Weil-smooth Topology on Schemes over $K$}

For any arbitrary scheme $Y$, let us recall the definition of the smooth site $Y_{sm}$.  The underlying category is the category of schemes which are smooth and locally of finite type over $Y$, and the coverings are the surjective families.  In \cite{vanhamel2004}, van Hamel illustrates the utility of the smooth site in the study of duality theorems; the cohomology groups coincide with those familiar from the \'etale site, but the internal hom functor is better suited to proving duality results.

This chapter is devoted to introducing a variant of the Weil-\'etale topology, the \textit{Weil-smooth} topology.  This definition is motivated by the definition of the Weil-\'etale topology given by Jiang in \cite{jiang2006}, and is related to the smooth topology in the same way that the Weil-\'etale topology is related to the \'etale topology.  As with the smooth site, the internal hom functor on the Weil-smooth site is more appropriate for a functorial approach to duality results.

\subsection{Definitions and Basic Properties}  

Throughout this section we fix a scheme $X$ which is smooth and finite type over $K$.

\begin{definition}
Let $\pi_1:X_L\rightarrow X$ and $\pi_2:X_L\rightarrow \Spec\ L$ be the projections.  We define the \textit{Weil-smooth} topology $W(X)$ to be the following Grothendieck topology:

\begin{itemize}

\item[(i)] The objects of $W(X)$ are the schemes which are smooth and locally of finite type over $X_L$.  That is, they are the objects of the smooth site of $X_L$.

\item[(ii)] A morphism $(V\stackrel{f}{\rightarrow} X_L) \rightarrow (Z\stackrel{g}{\rightarrow} X_L)$ of objects in $W(X)$, for connected $V$, is a map $\phi:V\rightarrow Z$ of schemes, such that (a) $\pi_1\circ g\circ \phi = \pi_1\circ f$, and (b) there exists $n\in\Z$ such that $\sigma^n\circ \pi_2\circ f = \pi_2 \circ g\circ \phi$, where $\sigma$ is the Frobenius automorphism of $L$.  If $V$ is not connected, we impose these conditions component-wise.

\item[(iii)] The coverings in $W(X)$ are the surjective families.

\end{itemize}
We let $X_W$ denote $X$ endowed with the Weil-smooth topology.
\end{definition}

If $Y$ is a scheme, $G$ is a discrete group of automorphisms of $Y$, and $F\in \mathcal{S}(Y_{sm})$, we say that $G$ \textit{acts on} $F$ if there are morphisms $F\rightarrow \tau_*F$ of sheaves for all $\tau\in G$, compatible in the obvious sense with the multiplication in $G$.  We denote the category of sheaves on $Y_{sm}$ which carry a $G$-action by $\mathcal{S}(Y_{sm})_G$.

\begin{proposition}
The category $\mathcal{S}(X_W)$ is equivalent to the category $\mathcal{S}(X_{L,sm})_{\frak{w}}$ 
\end{proposition}

\begin{proof}
This is the same proof as the analogous result for Weil-\'etale sheaves on schemes over finite fields; see Proposition 2.2 of \cite{lichtenbaum2005}.  
\end{proof}




Let $\bar{X}=X\times_K\bar{L}$ and let $\phi:\bar{X}\rightarrow X_L$.  In a slight abuse of notation, for any $F\in\mathcal{S}(X_W)=\mathcal{S}(X_{L,sm})_{\frak{w}}$, we denote the pullback to $\bar{X}$ also by $F$.

\begin{proposition}
The group $F(\bar{X})$ is a $W$-module, and $H^0(W, F(\bar{X})) = H^0(\frak{w},F(X_L))$.
\end{proposition}

\begin{proof}
It is clear that $F(\bar{X})$ is an $I$-module; we must show that this action extends to all of $W$.  Let $a\in W$ be a preimage of the Frobenius element $\sigma$ of $\frak{w}$.  We have $F(\bar{X})=\varinjlim_{L'/L}F(X_{L'})$ where the limit ranges over all of the finite extensions $L'$ of $L$.  

For a fixed finite extension $L'/L$, let $I'\subseteq I$ be the open subgroup corresponding to $L'$, and let $a'\in W/I'$ be the image of $a$.  There is a short exact sequence
$$1\rightarrow \Gal(L'/L)\rightarrow W/I'\rightarrow \frak{w}\rightarrow 1$$
of groups, and the choice of $a'$ determines a splitting $W/I'\simeq \Gal(L'/L)\rtimes \frak{w}$.
Because $F$ is endowed with an action of $\frak{w}$, the groups $\Gal(L'/L)$ and $\frak{w}$ both act on $F(X_{L'})$, and it is easy to check these actions are compatible with the decomposition of $W/I'$ as a semi-direct product.  Therefore $W/I'$ acts on $F(X_{L'})$, and passing to the limit over all $L'$ shows that $F(\bar{X})$ is a $W$-module.  One can verify easily that a different choice of $a$ gives an isomorphic $W$-module.

The second statement follows from the description of the $W$-module structure, and the fact that $H^0(\Gal(L'/L),F(X_{L'}))=F(X_L)$ for any finite extension $L'/L$.
\end{proof}

\begin{definition}
Let $F\in\mathcal{S}(X_W)$.  Define the \textit{$i^{th}$ Weil-smooth cohomology group} of $X$ with coefficients in $F$ by setting $\Gamma_X(F)=H^0(X_W,F)=H^0(\frak{w},F(X_L))$, and letting $H^i(X_W,F)$ be the $i^{th}$ right derived functor of $H^0(X_W,-)$ applied to $F$.  If $F$ is a complex of sheaves in $\mathcal{D}(X_W)$, we let $R\Gamma_X(F)$ denote the derived functor of $\Gamma_X$ applied to $F$.
\end{definition}

\begin{theorem}
There are spectral sequences
\begin{itemize}
\item[(i)] $H^p(\frak{w},H^q(X_{L,sm},F))\Rightarrow H^{p+q}(X_W,F)$, and
\item[(ii)] $H^p(W,H^q(\xbar_{sm}, F))\Rightarrow H^{p+q}(X_W,F)$
\end{itemize}
for any $F\in \mathcal{S}(X_W)$.
\end{theorem}

\begin{proof}
To establish the first spectral sequence, note that we can factor $\Gamma_X$ as $F\mapsto F(X_L)\mapsto H^0(\frak{w},F(X_L))$.  The functor $F\mapsto F(X_L)$ preserves injectives, since it has as exact left adjoint the functor $\frak{w}$-Mod $\rightarrow \mathcal{S}(X_{L,sm})_{\frak{w}}$ which takes a $\frak{w}$-module to the corresponding locally constant sheaf.  Part (i) is now just the spectral sequence of composite functors.

To see that the second spectral sequence holds, note that for any sheaf $F\in\mathcal{S}(X_{L,sm})_{\frak{w}}$, we have a spectral sequence
$$H^p(I,H^q(\bar{X}_{sm}, F))\Rightarrow H^{p+q}(X_{L,sm},F)$$
of $\frak{w}$-modules; see Theorem III.2.20 and Remark III.2.21(a) of \cite{milne1980}.  In derived category language, we have an isomorphism $R\Gamma_I\circ R\Gamma_{\bar{X}}(F) \simeq R\Gamma_{X_L}(F)$.  Applying $R\Gamma_{\frak{w}}$ to both sides of this isomorphism, we obtain
\begin{eqnarray*}
R\Gamma_W\circ R\Gamma_{\bar{X}}(F)&\simeq& R\Gamma_{\frak{w}}\circ R\Gamma_I\circ R\Gamma_{\bar{X}}(F)\\
&\simeq& R\Gamma_{\frak{w}}\circ R\Gamma_{X_L}(F)\\
&\simeq& R\Gamma_X(F).
\end{eqnarray*}
The first isomorphism is simply the spectral sequence coming from the group extension $1\rightarrow I\rightarrow W\rightarrow \frak{w}\rightarrow 1$, and the third isomorphism is from part (i).  The isomorphism $R\Gamma_X(F)\simeq R\Gamma_W\circ R\Gamma_{\bar{X}}(F)$ defines the desired spectral sequence, as in Corollary 10.8.3 of \cite{weibel1994}.
\end{proof}

\begin{corollary}
Suppose that $F\in\mathcal{S}(K_W)$ is defined by a smooth commutative group scheme defined over $K$.  Then $H^p(K_W,F)=H^p(W,F(\bar{L}))$.  
\end{corollary}

\begin{proof}
This is immediate from the second spectral sequence of the previous theorem, and the fact that $H^q(\bar{L}_{sm},F)=0$ for all $q\geq 1$ and all such $F$.  This last statement follows, for example, from the fact that smooth and \'etale cohomology agree for sheaves given by smooth commutative group schemes; see \S 1.2 of \cite{vanhamel2004}.
\end{proof}

\subsection{Smooth and Weil-Smooth Sheaves}

Adjusting notation slightly, let $\rho:X_L\rightarrow X$ be the natural map.  Following Proposition 2.4 of \cite{lichtenbaum2005}, we use $\rho$ to define a pair of adjoint functors.  For any $F\in\mathcal{S}(X_{sm})$, the sheaf $\rho^*F\in\mathcal{S}(X_L)$ carries a natural $\frak{w}$-action, and hence $\rho^*$ defines a pullback functor $\rho^*:\mathcal{S}(X_{sm})\rightarrow \mathcal{S}(X_W)$.  We define $\rho_*^\frak{w}:\mathcal{S}(X_W)\rightarrow \mathcal{S}(X)$ by the rule $(\rho_*^\frak{w}G)(U) = H^0(\frak{w},G(U_L))$ for any $G\in\mathcal{S}(X_W)$.

\begin{proposition}\label{someprops}
Let $\rho^*$ and $\rho_*^\frak{w}$ be as above.  We have the following:
\begin{itemize}
\item[(i)]  $\rho^*$ is left adjoint to $\rho_*^\frak{w}$.
\item[(ii)] $\rho^*$ is exact, and therefore $\rho_*^\frak{w}$ preserves injectives.
\item[(iii)] If $G/X$ is a smooth commutative group scheme which is locally of finite type, then there is natural isomorphism $\rho^*G\stackrel{\sim}{\rightarrow} G_L$ in $\mathcal{S}(X_W)$.
\item[(iv)] For any $G\in\mathcal{S}(X_{sm})$, there is a canonical map $G\rightarrow \rho_*^\frak{w}\rho^*G$, which is an isomorphism when $G$ is representable by a smooth commutative group scheme which is locally of finite type.
\item[(v)] For any smooth sheaf $F\in\mathcal{S}(X_{sm})$, there is a map of spectral sequences from
$$H^i(G,H^j(X_{\bar{K},sm},F))\Rightarrow H^{i+j}(X_{sm},F)$$
to
$$H^i(W,H^j(\bar{X},\phi^*\rho^*F))\Rightarrow H^{i+j}(X_W,\rho^*F).$$
\end{itemize}
\end{proposition}

\begin{proof}
Part (i) is proved in the usual manner, and part (ii) holds because pullback is always exact.  To see part (iii), we imitate the proof of Chapter II, Remark 3.1(d) of \cite{milne1980}.  Let $G/X$ be a smooth commutative group scheme, which we identify with the sheaf it defines on $X_{sm}$, and let $F\in\mathcal{S}(X_W)$ be any Weil-smooth sheaf.  By definition of $\rho_*^\frak{w}$ we have $H^0(\frak{w},F(G_L))\stackrel{\sim}{\rightarrow} (\rho_*^\frak{w}F)(G).$ By basic properties of representable sheaves, this implies that
$$\Hom_{X_W}(G_L,F)\stackrel{\sim}{\rightarrow} \Hom_{X_{sm}}(G,\rho_*^\frak{w}F)$$
and the result follows by uniqueness of adjoints.

The map described in part (iv) is the map induced by the adjunction map $G\rightarrow \rho_*\rho^*G$, the image of which lands in $\rho_*^\frak{w}\rho^*G$.  When $G$ is given by a smooth commutative group scheme, we have by part (iii) that $\rho^*G=G_L$, and thus the map is the natural map $G(U)\rightarrow H^0(\frak{w},G_L(U_L))$.  It is easy to see that this induces an isomorphism of sheaves.  The map of spectral sequences in part (v) is simply the map induced by the inclusion $W\rightarrow G$ and the projection $\bar{X}\rightarrow X_{\bar{K}}$.  
\end{proof}

The exact functor $\rho^*$ extends naturally to a functor $\rho^*:\mathcal{D}(X_{sm})\rightarrow \mathcal{D}(X_W)$ between the corresponding derived categories.    If $F\in\mathcal{D}(X_{sm})$ is a complex such that $H^i(F)$ is representable by some smooth commutative group scheme $G^i/X$ for all $i$, then by the exactness of $\rho^*$ we have $H^i(\rho^*F)=\rho^*G^i = G^i_L$.  If $F\in\mathcal{S}(X_{sm})$ is representable by a smooth commutative group scheme which is locally of finite type, then we will often simply write $F$ instead of $\rho^*F$ for the corresponding Weil-smooth sheaf it defines.

\subsection{Internal Hom and Pairings}

Let $X/K$ be a smooth scheme of finite type over $K$, and let $F,F'\in\mathcal{S}(X_W)$.  We define the Weil-smooth sheaf hom by $\underline{\Hom}(F,F') = \underline{\Hom}_{X_W}(F,F') = \underline{\Hom}_{X_{L,sm}}(F,F'),$ which carries a natural $\frak{w}$-action.  The functor $\underline{\Hom}(F,-)$ is left exact, and we denote by $R\underline{\Hom}(F,-)$ its derived functor.

\begin{lemma}\label{pullbackhom}
For any $F,G\in\mathcal{D}(K_{sm})$, there is a canonical map
$$
\Phi(F,G):\rho^*R\underline{\Hom}_{K_{sm}}(F,G)\rightarrow R\underline{\Hom}(\rho^*F,\rho^*G)
$$
in $\mathcal{D}(K_W)$.  If $G=\G_m$ and $F$ is a torus, an abelian variety, or a free finitely generated group scheme, then $\Phi(F,G)$ is an isomorphism.
\end{lemma}

\begin{proof}
By standard adjointness properties (see \cite{weibel1994}, Chapter 10.7.1), we have identifications
\begin{eqnarray*}
&& \Hom_{\mathcal{D}(K_W)}(\rho^*R\underline{\Hom}_{K_{sm}}(F,G),R\underline{\Hom}(\rho^*F,\rho^*G))\\
&=& \Hom_{\mathcal{D}(K_{sm})}(R\underline{\Hom}_{K_{sm}}(F,G),\rho_*^\frak{w}R\underline{\Hom}(\rho^*F,\rho^*G))\\
&=& \Hom_{\mathcal{D}(K_{sm})}(R\underline{\Hom}_{K_{sm}}(F,G),R\underline{\Hom}_{K_{sm}}(F,\rho_*^\frak{w}\rho^*G)).
\end{eqnarray*}
We define $\Phi(F,G)$ to be the map induced by the canonical map $G\rightarrow \rho_*^\frak{w}\rho^*G$ of (\ref{someprops} (iv)), which is the identity map when $G=\G_m$.

Now set $G=\G_m$, and suppose that $F=M$ is a free finitely generated commutative group scheme.  In this case $R\underline{\Hom}_{K_{sm}}(M,\G_m) = \underline{\Hom}_{K_{sm}}(M,\G_m) = T$ is a torus (see \cite{vanhamel2004}, Corollary 1.4).   Thus $\Phi(M,\G_m)$ is the natural map
$$\Phi(M,\G_m) : T_L\rightarrow \underline{\Hom}_{K_W}(M_L,\G_{m,L})$$
which is clearly an isomorphism.  For $F=T$ a torus, the same argument, with the roles of $T$ and $M$ reversed, shows that $\Phi(F,\G_m)$ is an isomorphism.

For $F=A$ an abelian variety we have $R\underline{\Hom}_{K_{sm}}(A,\G_m)=\underline{\Ext}_{K_{sm}}^1(A,\G_m) = A^t[-1]$ (see \cite{vanhamel2004}, Corollary 1.4), where $A^t$ is the dual abelian variety of $A$.  The map $\Phi(A,\G_m)$ now reads
$$\Phi(A,\G_m):A^t_L[-1]\rightarrow \underline{\Ext}_{K_W}^1(A_L,\G_{m,L})$$
which is an isomorphism by the compatibility of the Barsotti-Weil formula with base change.
\end{proof}

 If $F\in\mathcal{D}(K_W)$ is any bounded complex of sheaves, we define its \textit{Cartier Dual} by $F^D:=R\underline{\Hom}(F,\G_m)$.  If $G\in\mathcal{D}(K_{sm})$ we define $G^{D_{sm}}:=R\underline{\Hom}_{K_{sm}}(G,\G_m)$.  The previous proposition essentially says that if we restrict ourselves to tori and their cocharacter groups, and abelian varieties, we have $\rho^*((-)^{D_{sm}}) = (\rho^*(-))^D$.

\begin{proposition}\label{duals}  Let $T$ be a torus over $K$ with cocharacter group $M$, and $A$ an abelian variety over $K$ with dual abelian variety $A^t$.  Then we have the following natural isomorphisms in $\mathcal{D}(K_W)$:
\begin{eqnarray}\label{inthom}
M^D \simeq  T,\ T^D \simeq M,\ A^D \simeq A^t[-1].
\end{eqnarray}
\end{proposition}

\begin{proof}
By Corollary 1.4 of \cite{vanhamel2004}, the isomorphisms we are trying to demonstrate hold in $\mathcal{D}(K_{sm})$.  Applying $\rho^*$ to van Hamel's isomorphisms and using (\ref{pullbackhom}), we arrive at the corresponding isomorphisms in $\mathcal{D}(K_W)$.
\end{proof}

Let us return now to an arbitrary smooth scheme $X/K$ of finite type over $K$, and let $F\in\mathcal{D}(X_W)$.  There is a Yoneda pairing
$$
F'\otimes^L R\underline{\Hom}(F',F)\rightarrow F
$$
for any $F'\in\mathcal{D}(X_W)$.  Suppose that $H^n(X_W,F)\neq 0$, but $H^m(X_W,F)$ vanishes for all $m>n$.  Then by applying $R\Gamma_X(-)$ and projecting, we arrive at a pairing
\begin{equation}\label{yoneda}
R\Gamma_X(F')\otimes^L R\Gamma_X(R\underline{\Hom}(F',F))\rightarrow H^n(X_W,F)[-n]
\end{equation}
in $\mathcal{D}(\Z)$, which we will also call the Yoneda pairing.  If $X=\Spec\ K$, $F'=T$ is a torus with cocharacter group $M$, and $F=\G_m$, then by the above proposition we arrive at the pairing $R\Gamma_K(T)\otimes^L R\Gamma_K(M)\rightarrow \Z[-1]$ of (\ref{fgweilduality}).

\begin{proposition}\label{dpairing}
Let $\pi:X\rightarrow K$ be a smooth, projective curve over a $p$-adic field $K$.  Then in $\mathcal{D}(K_W)$, there is a canonical isomorphism 
$$
R\underline{\Hom}(R\pi_*\G_m,\G_m[-1])\stackrel{\sim}{\rightarrow} R\pi_*\G_m
$$
which induces a pairing
$$
R\pi_*\G_m\otimes^L R\pi_*\G_m\rightarrow \G_m[-1].
$$
\end{proposition}

\begin{proof}

Quite generally, Let $S$ be a scheme, and let $\pi:X\rightarrow S$ be a smooth, proper curve over $S$.  Deligne, in \cite{deligne1973}, has constructed an isomorphism
\begin{equation}\label{deligneiso}
\tau_{\leq 1}R\underline{\Hom}(\tau_{\leq 1}R\pi_*\G_m[1],\G_m)\rightarrow \tau_{\leq 1}R\pi_*\G_m
\end{equation}
of sheaves on $S_{fppf}$.  Let us make this isomorphism explicit when $S$ is the spectrum of a field $F$ of characteristic zero.  As noted in \cite{vanhamel2004}, in this case (\ref{deligneiso}) holds even on the smooth site of $S$.

Let $F(X)$ denote the function field of $X$.  Then in $\mathcal{D}(X_{sm})$, the complex $\G_m[1]$ is isomorphic to the complex $F(X)^\times[1]\rightarrow\Div_X$ where the map takes a function to its divisor.  Applying $R\pi_*$ to this complex, we see that $\pi_*\G_m[1]=\G_m[1]$ and $R^1\pi_*\G_m[1]=\Pic_X$, where by $\Pic_X$ we mean the sheaf on $F_{sm}$ defined by the Picard scheme of $X/F$.

For $U$ smooth over $F$, let $\Z[X(U)]$ be the free abelian group on the set of morphisms from $U$ to $X$ over $F$, and let $\Z^X$ be the sheaf on $F_{sm}$ associated to $U\mapsto \Z[X(U)]$.  There is a map $\Z^X\rightarrow R\pi_*\G_m[1]$ in $\mathcal{D}(F_{sm})$, given by taking a morphism $U\rightarrow X$ to its divisor in $\Div(X\times_F U)$.  Applying $R\underline{\Hom}_{F_{sm}}(-,\G_m)$ to this map, we arrive at Deligne's isomorphism
$$
R\underline{\Hom}_{F_{sm}}(R\pi_*\G_m[1],\G_m)\stackrel{\sim}{\rightarrow} R\underline{\Hom}_{F_{sm}}(\Z^X,\G_m)=R\pi_*\G_m,
$$
which encodes the auto-duality of the Jacobian of $X$, and the duality between the sheaves $\Z$ and $\G_m$.  The identity $R\underline{\Hom}_{F_{sm}}(\Z^X,\G_m)=R\pi_*\G_m$ follows from Yoneda's Lemma (this is where we use the fact that $X$ is smooth over $S$).

When $F=K$ Deligne's isomorphism reads $(R\pi_*\G_m[1])^{D_{sm}}\stackrel{\sim}{\rightarrow} R\pi_*\G_m$.  The cohomology sheaves of the complex $R\pi_*\G_m$ are all free finitely generated group schemes, tori, or abelian varieties, or extensions of such sheaves.  Thus (\ref{pullbackhom}) and (\ref{duals}) imply that
$$
(\rho^*R\pi_*\G_m)^D\simeq \rho^*(R\pi_*\G_m[1])^{D_{sm}} \simeq \rho^*R\pi_*\G_m
$$
which is the desired canonical isomorphism.  The pairing $R\pi_*\G_m\otimes^L R\pi_*\G_m\rightarrow \G_m[-1]$ is induced by standard adjoint properties and a degree shift.
\end{proof}

On applying $R\Gamma_K(-)$ to each term in the pairing (\ref{dpairing}) and composing with the map $R\Gamma_K(\G_m[-1])\rightarrow \Z[-2]$, we arrive at a pairing
\begin{equation}\label{dpairing2}
R\Gamma_X(\G_m)\otimes^L R\Gamma_X(\G_m)\rightarrow \Z[-2]
\end{equation}
in the derived category of abelian groups.  We let
\begin{equation}\label{lpairing}
\lambda(X):R\Gamma_X(\G_m)\rightarrow R\Hom(R\Gamma_X(\G_m),\Z[-2])
\end{equation}
be the induced map (by the symmetry of Deligne's pairing, both induced maps are the same).  The main theorem of this chapter will describe to what extent $\lambda(X)$ is an isomorphism.

\subsection{Cohomology of $K$ with Abelian Variety Coefficients}

Studying the map $\lambda(X)$ will require us to understand the cohomology of $W$ acting on the $\bar{L}$-points of the Jacobian of $X$.  Therefore, it will be useful to establish a Weil group analogue of Tate's duality theorem for abelian varieties over local fields, found in \cite{tate1957}.  

To that end, we devote this section to the pairing (\ref{yoneda}) when $F=\G_m$ and $F'=A$ is an abelian variety over $K$.  Using (\ref{inthom}) and shifting, we see that the Yoneda pairing induces a pairing
$$
R\Gamma_K(A)\otimes^L R\Gamma_K(A^t)\rightarrow \Z
$$
which is equal to the pairing induced by the biextension map $A\otimes^L A^t\rightarrow \G_m[1]$.  We let $\tau(A)$ denote the induced map,
\begin{equation}\label{biext2}
\tau(A):R\Gamma_K(A^t)\rightarrow R\Hom(R\Gamma_K(A),\Z).
\end{equation}
Our duality theorem for abelian varieties will describe to what extent $\tau(A)$ is an isomorphism.

\begin{lemma}
$H^i(W,A)=0$ for $i\neq 0,1$.
\end{lemma}
\begin{proof} The only non-trivial assertion is that $H^2(W,A)=0$.  For this, one can use the same proof as for the vanishing of $H^2(G,A(\bar{K}))$; see Chapter II, \S 5.3, Proposition 16 of \cite{serre2002}.
\end{proof}

\begin{lemma}\label{homvanish}

Let $Y$ be any of the following groups: $\Z_p,\mathcal{O}_K,U_K,A(K)$.  Then $\Hom(Y,\Z)$ is zero.

\end{lemma}

\begin{proof}

First consider the case of $Y=\Z_p$.  Let $f:\Z_p\rightarrow \Z$ be a non-zero homomorphism; since the only non-trivial subgroups of $\Z$ are isomorphic to $\Z$, we may assume $f$ is surjective.  Composing with the surjection $\Z\rightarrow \Z/p^n\Z$, we see that $f$ induces a surjection $\Z_p\rightarrow \Z/p^n\Z$.  This latter map must factor through $\Z_p/p^n\Z_p$, and therefore $f$ induces a surjection $\Z_p/p^n\Z_p\rightarrow \Z/p^n\Z.$ As these two finite groups have the same order, this is an isomorphism.  It follows that $\ker(f)\subseteq \bigcap_n p^n\Z_p=0$, and hence $f$ is injective.  Thus $f$ is an isomorphism, which is a contradiction.

Now suppose that $Y=\mathcal{O}_K$.  The vanishing of $\Hom(\mathcal{O}_K,\Z)$ follows immediately from the fact that $\mathcal{O}_K$ is a free $\Z_p$-module of rank equal to $[K:\Q_p]$.  The result for $Y=U_K$ follows from the fact that $U_K$ contains a subgroup of finite index isomorphic to $\mathcal{O}_K$ as abstract abelian groups.  Similarly, $A(K)$ contains a finite index subgroup isomorphic to $\dim A$ copies of $\mathcal{O}_K$.
\end{proof}

\begin{lemma}\label{abvarcomp}
The restriction map $R\Gamma_G(A)\rightarrow R\Gamma_W(A)$ is an isomorphism in $\mathcal{D}(\Z)$.
\end{lemma}
\begin{proof}
We must show that the maps $H^i(G,A)\rightarrow H^i(W,A)$ are isomorphisms for all $i\geq 0$.  In light of (\ref{weilgalalgcomp}), we only need to show that $H^1(W,A)$ is torsion.  Recall from (\ref{weilses}) that the group $H^{1}(W,A)$ fits into the exact sequence
$$0\rightarrow H^{1}(\frak{w},A(L))\rightarrow H^{1}(W,A)\rightarrow H^{0}(\frak{w},H^{1}(N,A))\rightarrow 0.$$
The group on the right is torsion, because it is a subgroup of a Galois cohomology group. Thus we are reducing to showing that $H^1(\frak{w},A(L))$ is torsion.

Let $\mathcal{A}/\mathcal{O}_{L}$ be the N\'eron model for $A$ over the ring of integers of $L$ ($\mathcal{A}$ is the base change to $\mathcal{O}_{L}$ of the N\'eron model for $A/K$; see \cite{bln1980}, Theorem 7.2.1 and Corollary 2.).  Let $\mathcal{A}_{0}\subseteq\mathcal{A}$ be the subscheme whose special fiber is the identity component of the special fiber of $\mathcal{A}$, and whose generic fiber is $A_{L}$.  We have an exact sequence of $\frak{w}$-modules,
 $$0\rightarrow \mathcal{A}_{0}(\mathcal{O}_{L})\rightarrow \mathcal{A}(\mathcal{O}_{L})\rightarrow\pi_{0}(\bar{k})\rightarrow 0,$$
where $\pi_{0}(\bar{k})$ is the group of connected components of the special fiber of $\mathcal{A}$.  This sequence is exact by Hensel's Lemma; see Proposition I.3.8 of \cite{milne1980}.  Taking cohomology gives a short exact sequence
$$H^{1}(\frak{w},\mathcal{A}_{0}(\mathcal{O}_{L}))\rightarrow H^{1}(\frak{w},\mathcal{A}(\mathcal{O}_{L}))\rightarrow H^{1}(\frak{w},\pi_{0}(\bar{k}))\rightarrow 0,$$
and it follows from Proposition 3 of \cite{greenberg} that $H^{1}(\frak{w},\mathcal{A}_{0}(\mathcal{O}_{L}))=0$.  Since $\mathcal{A}(\mathcal{O}_{L})=A(L)$, we have that $H^{1}(\frak{w},A(L))=H^{1}(\frak{w},\pi_{0}(\bar{k}))$, which is finite.
\end{proof}

\begin{theorem}\label{weilabvarduality}
The map $\tau(A)$ of (\ref{biext2}) has the following properties:  
\begin{itemize}

\item[(i)] $\tau(A)^0:A^t(K)\rightarrow \Ext(H^1(W,A),\Z)=H^1(W,A)^*$ is an isomorphism of profinite groups.

\item[(ii)] $\tau(A)^1:H^1(W,A^t)\rightarrow \Ext(A(K),\Z)$ induces an isomorphism of $H^1(W,A^t)$ with the torsion subgroup $A(K)^*$ of $\Ext(A(K),\Z)$.

\end{itemize}
The cohomology of both complexes vanishes outside of degrees 0 and 1.  In particular, $\tau(A)^i$ is injective for all $i$.

\end{theorem}

\begin{proof}
By (\ref{rhomz}) and (\ref{homvanish}), the maps $\tau(A)^i$ reduce to maps
$$
\tau(A)^i:H^i(W,A^t)\rightarrow \Ext(H^{1-i}(W,A),\Z).
$$
The group $H^1(W,A)$ is torsion, hence has no non-zero maps to $\Q$.  Part (i) of the theorem now follows from (\ref{abvarcomp}) and Tate's duality theorem on abelian varieties over local fields (see the main theorem of \cite{tate1957}).  The profinite group $A(K)$ admits no continuous maps to $\Q$, so $A(K)^*$ injects into $\Ext(A(K),\Z)$.  Part (ii) of the theorem now follows again from (\ref{abvarcomp}) and Tate's theorem.
\end{proof}

\subsection{Duality for Weil-smooth Cohomology of Curves}

Before stating our Weil-smooth duality theorem for curves, we would like to remind the reader of Lichtenbaum's  duality theorem for curves over $p$-adic fields, and van Hamel's approach to its construction and proof of non-degeneracy.  As always, let $\pi:X\rightarrow K$ be a smooth, projective, geometrically connected curve.

\begin{theorem} (Lichtenbaum, \cite{lichtenbaum1969})  There are natural pairings
$$H^i(X_{sm},\G_m)\otimes H^{3-i}(X_{sm},\G_m)\rightarrow \QZ$$
which induce isomorphisms $H^i(X_{sm},\G_m)\otimes\hat{\Z}\rightarrow H^{3-i}(X_{sm},\G_m)^*$ for all $i$, where $H^i(X_{sm},\G_m)$ has the natural topology coming from that on $K$.
\end{theorem}

Lichtenbaum defines his pairing by explicitly evaluating representatives of the Brauer group on divisor classes.  Since our objects live in the derived category where the notion of ``element'' does not make sense, van Hamel's  functorial approach adapts better to the Weil-smooth situation.

Let $F=R\pi_*\G_m\in\mathcal{D}(X_{sm})$, so that $R\Gamma_{K_{sm}}(F) = R\Gamma_{X_{sm}}(\G_m)$.  In \cite{vanhamel2004}, van Hamel's approach to Lichtenbaum's duality theorem is to put an ``ascending filtration'' on $F$.  That is, he defines complexes $F_i$ and constructs a series of morphisms $0\rightarrow F_0\rightarrow F_1\rightarrow F_2 = F$ in $\mathcal{D}(K_{sm})$.  For each $i\geq 0$, van Hamel defines the $i^{th}$ graded piece $G_i$ to be the mapping cone of $F_{i-1}\rightarrow F_i$, yielding an exact triangle $F_{i-1}\rightarrow F_i\rightarrow G_i\rightarrow F_{i-1}[1]$ in $\mathcal{D}(K_{sm})$.

The sheaf $F_0$ is defined by $F_0:=H^0(F)=\G_m$, and $F_1$ is defined to be the mapping cone of the composite $F\rightarrow \Pic_X[-1]\stackrel{\deg}{\rightarrow} \Z[-1]$.  Thus there are triangles
$$F_1\rightarrow F\rightarrow \Z[-1]\rightarrow F_1[1]\quad \text{and}\quad \G_m\rightarrow F_1\rightarrow \Pic_X^0[-1]\rightarrow \G_m[1]$$
in $\mathcal{D}(K_{sm})$.  The graded pieces $G_i$ are the given by
$$
G_i = \left\{
\begin{array}{cl}
\G_m & i = 0\\
\Pic_X^0[-1] & i=1\\
\Z[-1] & i = 2\\
0 & i \geq 3.
\end{array}
\right.
$$
The filtration on $F$ induces a ``descending filtration'' $F^{D_{sm}}=F_2^{D_{sm}}\rightarrow F_1^{D_{sm}}\rightarrow F_0^{D_{sm}}\rightarrow 0$ on $F^{D_{sm}}$, and also induces triangles $G_i^{D_{sm}}\rightarrow F_i^{D_{sm}}\rightarrow F_{i-1}^{D_{sm}}\rightarrow G_i^{D_{sm}}[1]$ for all $i$.  The sheaves $G_i^{D_{sm}}$ are given by
$$
G_i^{D_{sm}} = \left\{
\begin{array}{cl}
\Z & i = 0\\
\text{Alb}_X & i = 1\\
\G_m[1] & i = 2\\
0 & i \geq 3.
\end{array}
\right.
$$
To prove the non-degeneracy and perfectness results of the pairing, van Hamel then uses duality theorems for finitely generated group schemes, tori, and abelian varieties to analyze the pairings
$$R\Gamma_{K_{sm}}(G_i)\otimes^L R\Gamma_{K_{sm}}(G_i^{D_{sm}})\rightarrow \QZ[-2],$$
and pieces together a duality theorem for $R\Gamma_{K_{sm}}(F)=R\Gamma_{X_{sm}}(\G_m)$ using the Five Lemma.  We will essentially copy this approach, by applying $\rho^*$ to van Hamel's filtration and exact triangles.

We can now state and prove our duality theorem for the Weil-smooth cohomology of curves.

\begin{theorem}\label{curveduality}
Let $X/K$ be a smooth, projective, geometrically connected curve over $K$, such that $X(K)\neq\emptyset$.  The map
$$
\lambda(X):R\Gamma_X(\G_m)\rightarrow R\Hom(R\Gamma_X(\G_m),\Z[-2])
$$
induced by the pairing (\ref{dpairing2}) has the following properties:
\begin{itemize}
\item[(i)] $\lambda(X)^i$ is an isomorphism for $i=0,1$.
\item[(ii)] $\lambda(X)^i$ is injective for $i=2,3$.
\end{itemize}
The cohomology of both complexes vanishes outside of degrees $0$ through $3$.
\end{theorem}

\begin{proof}
As above, let $F$ be the complex $R\pi_*\G_m$ considered on the smooth site of $K$, so that
$$ R\Gamma_X(\G_m) = R\Gamma_{K_W}(\rho^* R\pi_*\G_m)$$
Applying $\rho^*$ to van Hamel's filtration provides us with a filtration $0\rightarrow \rho^*F_0\rightarrow \rho^*F_1\rightarrow \rho^*F_2 = \rho^*F$ on $F$
which comes equipped with exact triangles $\rho^*F_{i-1}\rightarrow \rho^*F_i\rightarrow \rho^*G_i\rightarrow \rho^*F_{i-1}[1].$

Now we apply $\rho^*$ to the dual filtration, and note that by (\ref{duals}), $\rho^*$ commutes with the dualizing functors in the sense that $\rho^*(G_i^{D_{sm}}) = (\rho^*G_i)^D$ for all $i$.  Repeatedly applying the Five Lemma and (\ref{pullbackhom}) to van Hamel's exact triangles shows that $\rho^*(F_i^D) = (\rho^*F_i)^D$ for all $i$.  Now the Yoneda pairing induces pairings
$$
R\Gamma_K(\rho^*G_i)\otimes^L R\Gamma_K((\rho^*G_i)^D)\rightarrow \Z[-1]
$$
$$
R\Gamma_K(\rho^*F_i)\otimes^L R\Gamma_K((\rho^*F_i)^D)\rightarrow \Z[-1]
$$
in $\mathcal{D}(\Z)$ for all $i$.  In a slight abuse of notation, we suppress $\rho^*$ from now on.

Let
$$
\gamma_i:R\Gamma_K(G_i^D)\rightarrow R\Hom(R\Gamma_K(G_i),\Z[-1])
$$
be the induced map in $\mathcal{D}(\Z)$.  We will describe the maps $\gamma_i$ in terms of duality theorems we have already proven. 

From (\ref{almostduality}) we see that $\gamma_0=\eta(\Z)$ has the following properties: $\gamma_0^i$ is an isomorphism for $i\neq 2$, and $\gamma_0^2$ maps $H^2(W,\Z)$ isomorphically onto the torsion subgroup $U_K^*$ of $\Ext(K^\times,\Z)$.

Let $J_X$ be the Jacobian variety of $X$.  Any rational point of $X$ determines an embedding $X\hookrightarrow J_X$ defined over $K$, and thus a Weil-equivariant isomorphism $\Piczero_X(\bar{L})\rightarrow J_X(\bar{L})$.  Hence we can identify $\Piczero_X(\bar{L})$ with the $\bar{L}$-points of an abelian variety defined over $K$, and apply (\ref{weilabvarduality}) to $\Piczero_X$ and its dual abelian variety $\text{Alb}_X$.

From (\ref{weilabvarduality}), we see that $\gamma_1=\tau(\Piczero_X)$ has the following properties: $\gamma_1^0$ is an isomorphism which maps $\text{Alb}_X(K)$ isomorphically onto $H^1(W,\Piczero_X)^*$, and $\gamma_1^1$ is an injective map which maps $H^1(W,\text{Alb}_X)$ isomorphically onto the torsion subgroup $\Piczero_X(K)^*$ of $\Ext(\Piczero_X(K),\Z)$.  From (\ref{fgweilduality}) we see that $\gamma_2=\psi(\Z)[1]$ is an isomorphism.  

Now consider the maps
$$
\phi_i:R\Gamma_K(F_i^D)\rightarrow R\Hom(R\Gamma_K(F_i),\Z[-1])
$$
induced by the Yoneda pairing.  We can determine to what extent the maps $\phi_i$ are isomorphisms, by using the triangles which defined $G_i$ and $G_i^D$.  When $i=0$ one has $G_0=F_0=\G_m$, hence $\phi_0=\gamma_0=\eta(\Z)$ is the map of (\ref{almostduality}).  When $i=1$ we have a diagram
$$
\xymatrix{
R\Gamma_K(\text{Alb}_X)\ar[r]^-{\gamma_1}\ar[d] &R\Hom(R\Gamma_K(\Piczero_X[-1]),\Z[-1])\ar[d]\\
R\Gamma_K(F_1^D)\ar[r]^-{\phi_1}\ar[d] &R\Hom(R\Gamma_K(F_1),\Z[-1])\ar[d]\\
R\Gamma_K(F_0^D)\ar[r]^-{\phi_0} &R\Hom(R\Gamma_K(F_0^D),\Z[-1]).
}
$$
It follows that $\phi_1^0$ is an isomorphism, $\phi_1^1$ is injective, and $\phi_1^2$ maps $H^2(K_W,F_1^D)$ isomorphically onto the torsion subgroup $U_K^*$ of $\Ext^2(R\Gamma_K(F_1),\Z[-1])$.  The cohomology of all of these complexes vanishes outside of degrees $0$ through $2$.

When $i=2$ we have a diagram
$$
\xymatrix{
R\Gamma_K(\G_m[1])\ar[r]^-{\psi(\Z)[1]}_-\sim\ar[d] &R\Hom(R\Gamma_K(\Z[-1]),\Z[-1])\ar[d]\\
R\Gamma_K(F_2^D)\ar[r]^-{\phi_2}\ar[d] &R\Hom(R\Gamma_K(F_2),\Z[-1])\ar[d]\\
R\Gamma_K(F_1^D)\ar[r]^-{\phi_1} &R\Hom(R\Gamma_K(F_1),\Z[-1]).
}
$$
It follows that $\phi_2^{-1}=\psi(\Z)^0$ is an isomorphism from $K^\times$ to $\Hom(R\Gamma_K(F_2),\Z[-1])$, that $\phi_2^0$ is an isomorphism, that $\phi_2^1$ is injective, and that $\phi_2^2$ is an isomorphism from $H^2(K_W,F_2^D)$ to the torsion subgroup $U_K^*$ of $\Ext^2(R\Gamma_K(F_2),\Z[-1])$.

The theorem is now clear, once we recall that $\lambda(X)$ is the map induced by the isomorphism $F_2[1]^D\rightarrow F_2$ of Weil-smooth sheaves, the map $\phi_2$, and shifting degrees by one.
\end{proof}



\subsection{Comparison with Smooth Cohomology}

In this section we compare the duality theorem of the previous section with the main theorem of \cite{lichtenbaum1969}.  That the our pairing is compatible with the original pairing defined by Lichtenbaum follows from the results of \S 3.3 of \cite{vanhamel2004}.

\begin{proposition}
Suppose that $X/K$ is a smooth, proper variety over $K$, and let $F$ be a torsion sheaf in $\mathcal{S}(X_{sm})$.  Then the restriction map $R\Gamma_{X_{sm}}(F)\rightarrow R\Gamma_{X}(F)$ is an isomorphism in $\mathcal{D}(\Z)$.
\end{proposition}

\begin{proof}
By (\cite{milne1980}, Chapter VI, Corollary 2.6), the map $R\Gamma_{X_{\bar{K},sm}}(F)\rightarrow R\Gamma_{\bar{X}_{sm}}(F)$ is an isomorphism.  The result now follows from (\ref{weilgaltorsion}), since $R\Gamma_{X_{sm}}=R\Gamma_G\circ R\Gamma_{X_{\bar{K},sm}}$ and $R\Gamma_X=R\Gamma_W\circ R\Gamma_{\bar{X}_{sm}}$.
\end{proof}

By the previous proposition, smooth and Weil-smooth cohomology agree for the sheaf $\mu_n$, hence we can use Kummer sequences to study the restriction maps $H^i(X_{sm},\G_m)\rightarrow H^i(X_W,\G_m)$.  In particular there is a diagram
$$
\xymatrix{
0\ar[r] &H^i(X_{sm},\G_m)/n\ar[r]^-\delta\ar[d]^-{\text{res}^i/n} &H^{i+1}(X_{sm},\mu_n)\ar[r]\ar[d]^-\wr &H^{i+1}(X_{sm},\G_m)[n]\ar[r]\ar[d]^-{\text{res}^{i+1}[n]} &0\\
0\ar[r] &H^i(X_W,\G_m)/n\ar[r]^-\delta &H^{i+1}(X_W,\mu_n)\ar[r] &H^{i+1}(X_W,\G_m)[n]\ar[r] &0
}
$$
from which we deduce an isomorphism $\delta^i_n:\ker(\text{res}^{i+1}[n])\rightarrow \text{coker}(\text{res}^i/n)$ for any pair of integers $i,n$.  Passing to the limit over all $n$, we obtain a canonical isomorphism
\begin{equation}\label{deltai}
\delta^i:\ker(\text{res}^{i+1}|_{\text{tors}})\rightarrow \text{coker}(\text{res}^i\otimes 1)
\end{equation}
where $\text{res}^i\otimes 1$ is the obvious map $H^i(X_{sm},\G_m)\otimes\QZ\rightarrow H^i(X_W,\G_m)\otimes\QZ$.

\begin{proposition}
Let $X/K$ be a smooth, projective, geometrically connected curve over $K$ such that $X(K)\neq \emptyset$.  The restriction maps $\text{res}^i:H^i(X_{sm},\G_m)\rightarrow H^i(X_W,\G_m)$ are described by the following exact sequences:
\begin{eqnarray}
0\rightarrow H^1(X_{sm},\G_m)\stackrel{\text{res}^1}{\rightarrow} H^1(X_W,\G_m)\rightarrow H^1(W,\G_m)\rightarrow 0\label{weilsmoothcomp1}\\
0\rightarrow \Br(K)\rightarrow H^2(X_{sm},\G_m) \stackrel{\text{res}^2}{\rightarrow} H^1(W,\Piczero_X)\rightarrow 0\label{weilsmoothcomp2}\\
0\rightarrow H^1(W,\Z)\otimes\QZ\rightarrow H^3(X_{sm},\G_m)\stackrel{\text{res}^3}{\rightarrow} H^3(X_W,\G_m)\rightarrow 0. \label{weilsmoothcomp3}
\end{eqnarray}
\end{proposition}

\begin{proof}
The map of Hochschild-Serre spectral sequences computing smooth and Weil-smooth cohomology gives us a map of short exact sequences
$$
\xymatrix{
0\ar[r] &0\ar[r]\ar[d] &H^1(X_{sm},\G_m)\ar[r]^-\sim\ar[d]^-{\text{res}^1} &H^0(G,\Pic_X)\ar[r]\ar[d]^-\wr &0 \\
0\ar[r] &H^1(W,\G_m)\ar[r] &H^1(X_W,\G_m)\ar[r] &H^0(W,\Pic_X)\ar[r] &0
}
$$
coming from the long exact sequences of low degree.  That the non-zero map in the top row is an isomorphism follows from $X(K)\neq\emptyset$.  The existence of (\ref{weilsmoothcomp1}) follows by applying the Snake Lemma.

To prove the second existence of the second exact sequence, note that $H^2(X_{sm},\G_m)$ is a torsion group, so $\ker(\text{res}^2_{\text{tors}})=\ker(\text{res}^2)$.  We will show that there is a natural identification $\ker(\text{res}^2)=\Br(K)$.  The long exact sequences of low degree from the Hochschild-Serre spectral sequences  give us a map of short exact sequences
$$
\xymatrix{
0\ar[r] &\Br(K)\ar[r]\ar[d] &H^2(X_{sm},\G_m)\ar[r]\ar[d]^-{\text{res}^2} &H^1(G,\Pic_X)\ar[r]\ar[d] &0\\
0\ar[r] &0\ar[r] &H^2(X_W,\G_m)\ar[r]^-\sim &H^1(W,\Pic_X)\ar[r] &0.
}
$$
It follows from (\ref{weilgalcomp}) and (\ref{weilgalalgcomp}) that $H^1(G,\Pic_X)$ can be identified with the torsion subgroup of $H^1(W,\Pic_X)$ via the restriction map.  

On the other hand, consider the long exact sequence in $W$-cohomology of $0\rightarrow \Piczero_X\rightarrow\Pic_X\rightarrow \Z\rightarrow 0$.  Since $X(K)\neq\emptyset$, any rational point determines a Weil-equivariant degree $1$ divisor class on $\bar{X}$, hence the map $\deg:H^0(W,\Pic_X)\rightarrow \Z$ is surjective.  The relevant part of the long exact sequence now reads
$$
H^0(W,\Pic_X)\stackrel{\deg}{\rightarrow} \Z\stackrel{0}{\rightarrow} H^1(W,\Piczero_X)\rightarrow H^1(W,\Pic_X)\rightarrow \Z\rightarrow 0,
$$
and we have an identification $H^1(W,\Piczero_X)=H^1(W,\Pic_X)_{\text{tors}}$.  That (\ref{weilsmoothcomp1}) and (\ref{weilsmoothcomp2}) are exact is now clear.

The only remaining task is to identify the kernel of $\text{res}^3$.  By (\ref{deltai}) we have an identification $\ker(\text{res}^3)=\text{coker}(\text{res}^2\otimes1)$.  But as $H^2(X_{sm},\G_m)$ is torsion, this last cokernel can be identified with $H^2(X_W,\G_m)\otimes\QZ=H^1(W,\Pic_X)\otimes\QZ$.  The map on cohomology induced by the degree map gives an isomorphism of this last group with $H^1(W,\Z)\otimes\QZ$, since $H^1(W,\Piczero_X)\otimes\QZ=0$.
\end{proof}

With the above comparison theorem, we can reprove the main result of \cite{lichtenbaum1969} for curves $X/K$ containing a rational point.

\begin{theorem}
Suppose that $X/K$ is a smooth, projective, geometrically connected curve, such that $X(K)\neq\emptyset$.  Then the Lichtenbaum pairing $H^2(X_{sm},\G_m)\otimes H^1(X_{sm},\G_m)\rightarrow \QZ$ induces an isomorphism $H^2(X_{sm},\G_m)\rightarrow H^1(X_{sm},\G_m)^*$.
\end{theorem}

\begin{proof}
The map induced by the Lichtenbaum pairing fits into the diagram
$$
\xymatrix{
0\ar[r] &\Br(K)\ar[r]\ar[d]^-\wr &H^2(X_{sm},\G_m)\ar[r]\ar[d] &H^1(W,\Piczero_X)\ar[r]\ar[d]^-\wr &0\\
0\ar[r] &\Z^*\ar[r] &H^1(X_{sm},\G_m)^*\ar[r] &\Piczero_X(K)^*\ar[r] &0
}
$$
where the top row is the exact sequence of (\ref{weilsmoothcomp2}).  The result follows by the Five Lemma.
\end{proof}

\begin{acknowledgements}
The research presented here constitutes the bulk of the author's Ph.D. thesis, completed at the University of Maryland, College Park.  As such, the author would like to extend his sincerest gratitude towards his thesis adviser, Niranjan Ramachandran, for suggesting this topic, and for consistently providing helpful comments, suggestions, and guidance.

Thomas Geisser, Baptiste Morin, and Mathias Flach also deserve a great amount of the author's gratitutde, for many helpful comments and questions regarding the main results of this article.  Lastly, the author would like to thank Thomas Haines and Larry Washington for helpful discussions concerning his thesis.
\end{acknowledgements}

\bibliographystyle{amsalpha}
\bibliography{myrefs.bib}

\providecommand{\bysame}{\leavevmode\hbox to3em{\hrulefill}\thinspace}
\providecommand{\MR}{\relax\ifhmode\unskip\space\fi MR }
\providecommand{\MRhref}[2]{%
  \href{http://www.ams.org/mathscinet-getitem?mr=#1}{#2}
}
\providecommand{\href}[2]{#2}
\begin{thebibliography}{GHKR10}

\bibitem[BLR90]{bln1980}
Siegfried Bosch, Werner L{\"u}tkebohmert, and Michel Raynaud, \emph{{N\'eron
  Models}}, Ergebnisse der Mathematik und ihrer Grenzgebiete (3) [Results in
  Mathematics and Related Areas (3)], vol.~21, Springer-Verlag, Berlin, 1990.

\bibitem[Con02]{conrad2002}
Brian Conrad, \emph{{A Modern Proof of Chevalley's Theorem on Algebraic
  Groups}}, J. Ramanujan Math. Soc. \textbf{17} (2002), no.~1, 1--18.
  \MR{1906417 (2003f:20078)}

\bibitem[Del73]{deligne1973}
Pierre Deligne, \emph{{Expos\'e XVIII: La Formule de Dualit\'e Globale}}, {\rm
  SGA 4}, Lecture Notes in Math, vol. 305, Springer-Verlag, New York, 1973,
  pp.~481--587.

\bibitem[Fla08]{flach2008}
Mathias Flach, \emph{{Cohomology of Topological Groups with Applications to the
  Weil Group}}, Compos. Math. \textbf{144} (2008), no.~3, 633--656.

\bibitem[FM10]{flachmorin2010}
Mathias Flach and Baptiste Morin, \emph{{On the Weil-\'etale Topos of Regular
  Arithmetic Schemes}}, \url{http://arxiv.org/abs/1010.3833}, 2010.

\bibitem[Gei10]{geisser2010}
Thomas Geisser, \emph{{On Suslin's Singular Homology and Cohomology}},
  Documenta Mathematica, Extra Volume Suslin (2010), 223--249.

\bibitem[GHKR10]{haines}
Ulrich G\"ortz, Thomas~J. Haines, Robert~E. Kottwitz, and Daniel~C. Reuman,
  \emph{{Affine Deligne-Lusztig Varieties in Affine Flag Varieties}},
  Compositio Math. \textbf{146} (2010), 1339--1382.

\bibitem[Gre63]{greenberg}
Marvin~J. Greenberg, \emph{{Schemata over Local Rings. {II}}}, Ann. of Math.
  (2) \textbf{78} (1963), 256--266.

\bibitem[Jia06]{jiang2006}
Yongbin Jiang, \emph{{Weil-\'etale Topology over Local Rings}}, ProQuest LLC,
  Ann Arbor, MI, May 2006, Thesis (Ph.D.) -- Brown University.

\bibitem[Lic69]{lichtenbaum1969}
Stephen Lichtenbaum, \emph{{Duality Theorems for Curves over $p$-adic Fields}},
  Invent. Math. \textbf{7} (1969), 120--136.

\bibitem[Lic99]{lichtenbaum1999}
\bysame, \emph{{Duality Theorems for the Cohomology of Weil Groups}}, 1999.

\bibitem[Lic05]{lichtenbaum2005}
\bysame, \emph{{The Weil-\'etale Topology on Schemes over Finite Fields}},
  Compos. Math. \textbf{141} (2005), no.~3, 689--702.

\bibitem[Lic09]{lichtenbaum2009}
\bysame, \emph{{The {W}eil-\'etale Topology for Number Rings}}, Ann. of Math.
  (2) \textbf{170} (2009), no.~2, 657--683.

\bibitem[Mil80]{milne1980}
J.~S. Milne, \emph{{Etale Cohomology}}, Princeton Mathematical Series, vol.~33,
  Princeton University Press, Princeton, N.J., 1980.

\bibitem[Moo76]{moore3}
Calvin~C. Moore, \emph{{Group Extensions and Cohomology for Locally Compact
  Groups. {III}}}, Trans. Amer. Math. Soc. \textbf{221} (1976), no.~1, 1--33.

\bibitem[Mor11a]{morin20101}
Baptiste Morin, \emph{{The Weil-\'etale Fundamental Group of a Number Field
  I}}, Kyushu J. Math. \textbf{65} (2011), no.~1, 101--140.

\bibitem[Mor11b]{morin20102}
\bysame, \emph{{The Weil-\'etale Fundamental Group of a Number Field II}},
  Selecta Math. (N.S.) \textbf{17} (2011), no.~1, 67--137.

\bibitem[NSW08]{neukirch2008}
J{\"u}rgen Neukirch, Alexander Schmidt, and Kay Wingberg, \emph{{Cohomology of
  Number Fields}}, second ed., Grundlehren der Mathematischen Wissenschaften
  [Fundamental Principles of Mathematical Sciences], vol. 323, Springer-Verlag,
  Berlin, 2008.

\bibitem[Raj04]{rajan2004}
C.~S. Rajan, \emph{{On the Vanishing of the Measurable Schur Cohomology Groups
  of Weil Groups}}, Compos. Math. \textbf{140} (2004), no.~1, 84--98.

\bibitem[Ser79]{serre1979}
Jean-Pierre Serre, \emph{{Local Fields}}, Graduate Texts in Mathematics,
  vol.~67, Springer-Verlag, New York, 1979, Translated from the French by
  Marvin Jay Greenberg.

\bibitem[Ser02]{serre2002}
\bysame, \emph{{Galois Cohomology}}, english ed., Springer Monographs in
  Mathematics, Springer-Verlag, Berlin, 2002, Translated from the French by
  Patrick Ion and revised by the author.

\bibitem[Tat57]{tate1957}
John Tate, \emph{{W{C}-groups over {$\frak{p}$}-adic Fields}}, S\'eminaire
  {B}ourbaki, {V}ol.\ 4, Soc. Math. France, Paris, 1957, pp.~Exp.\ No.\ 156,
  265--277.

\bibitem[vH04]{vanhamel2004}
Joost van Hamel, \emph{{Lichtenbaum-Tate Duality for Varieties over $p$-adic
  Fields}}, J. Reine Angew. Math. \textbf{575} (2004), 101--134.

\bibitem[Wei94]{weibel1994}
Charles~A. Weibel, \emph{{An Introduction to Homological Algebra}}, Cambridge
  University Press, 1994.

\end{thebibliography}

\end{document}